\documentclass{amsart}

\usepackage{amssymb}
\usepackage{amsthm}
\usepackage{mathrsfs}
\usepackage{enumerate} 
\usepackage{color}

\usepackage{enumitem} %<----- testing this may  7 2012

\usepackage{rotating}

\usepackage[all]{xy}

\usepackage{fancyhdr}
\usepackage{xr-hyper}

\usepackage[colorlinks=true]{hyperref}

\newtheorem{theorem}{Theorem}[section] %was[subsection] before
\newtheorem{proposition}[theorem]{Proposition}
\newtheorem{lemma}[theorem]{Lemma}

\newtheorem{corollary}[theorem]{Corollary}

\theoremstyle{definition}
\newtheorem{definition}[theorem]{Definition}
\newtheorem{example}[theorem]{Example}

\newtheorem{notation}[theorem]{Notation}

\theoremstyle{remark}
\newtheorem{remark}[theorem]{Remark}
\newtheorem{remarks}[theorem]{Remarks}

\newtheorem{properties}[theorem]{Properties}

\numberwithin{equation}{section}

\newcounter{myenum}

\makeatletter
\renewcommand\thetheorem{\@arabic\c@section.\@arabic\c@theorem}
\newcounter{subtheorem}
 \@addtoreset{subtheorem}{theorem}
\renewcommand\thesubtheorem{\thetheorem.\@arabic\c@subtheorem}
\newcommand\subtheorem{\stepcounter{subtheorem}\par\noindent\protect\textbf{\textup{(\thesubtheorem)}}\quad}
\makeatother

\newcommand{\cal}{\mathcal}
\newcommand{\bb}{\mathbb}
\newcommand{\scr}{\mathscr}

\newcommand{\comment}[1]{}

\newcommand{\xym}{\xymatrix}

\newcommand{\into}{\hookrightarrow}

\newcommand{\A}{\mathbb{A}^1}

\DeclareMathOperator{\spec}{Spec}
\DeclareMathOperator{\id}{id}

\DeclareMathOperator{\Exc}{Exc}

\DeclareMathOperator{\lf}{l.f.}
\DeclareMathOperator{\Ind}{Ind}

\DeclareMathOperator{\hocolim}{hocolim}
\newcommand{\oline}{\overline}

\newcommand{\leftsub}[2]{{\vphantom{#2}}_{#1}{#2}}

\begin{document}

\title{Cohomology theories with supports}
\author{Joseph Ross}
\address{University of Southern California Mathematics Department, 3620 South Vermont Avenue, Los Angeles, California}
\urladdr{http://www-bcf.usc.edu/~josephr/}
\email{josephr@usc.edu}

%
%\subjclass[2010]{Primary xxAxx}

%\keywords{group actions}

%\date{the date}

%
\maketitle
%\begin{abstract}
%\end{abstract}

\begin{abstract} For $E$ a presheaf of spectra on the category of smooth $k$-schemes satisfying Nisnevich excision, we prove that the canonical map from the algebraic singular complex of the theory $E$ with quasi-finite supports to the theory $E$ with supports intersecting all the faces properly is a weak equivalence on smooth $k$-schemes that are affine or projective.  This establishes some cases of a conjecture of Marc Levine.
\end{abstract}

\section{Introduction} \label{sec intro}

Let $k$ be a field, let ${\bf{Sm}}/k$ denote the category of smooth separated $k$-schemes of finite type, and let ${\bf{Spt}}(k)$ denote the category of presheaves of spectra on ${\bf{Sm}} / k$.  A cohomology theory is a presheaf of spectra satisfying certain conditions.  Such a theory may be filtered according to the (co)dimension of a subscheme supporting a cohomology class.  We describe now two ways of obtaining such a filtration, both of which make use of the cosimplicial object $\Delta^\bullet$ in ${\bf{Sm}}/k$.

For $E \in {\bf{Spt}}(k)$ and $X \in {\bf{Sm}} / k$, first define $\overline{E}^Q(X) := \hocolim_{Z \in Q} E^Z(X \times \mathbb{A}^q)$, where $Q$ consists of cycles on $X \times \mathbb{A}^q$ that are quasi-finite and dominant over (a component of) $X$.  This is the spectrum with quasi-finite supports.  Now define $E^Q$ to be the algebraic singular complex of the presheaf $\oline{E}^Q$, that is, $E^Q(X) = \oline{E}^Q(\Delta^\bullet \times X) := | n \mapsto \oline{E}^Q(\Delta^n \times X)|$.  This construction is functorial for all morphisms of schemes: $E^Q \in {\bf{Spt}}(k)$.

For $E \in {\bf{Spt}}(k)$ and $X \in {\bf{Sm}} / k$, we set $E^{(q)}(X,n) := \hocolim_{Z \in S^{(q)}} E^Z( \Delta^n \times X)$, where $S^{(q)}$ consists of those cycles on $\Delta^n \times X$ intersecting all faces in codimension at least $q$, i.e., properly.  This gives a simplicial spectrum, and we may form its total spectrum $E^{(q)}(X) :=| n \mapsto E^{(q)}(X,n) |$.  This construction is the evident generalization of the cycles used by Bloch \cite{Blhigher} to define higher Chow groups.  Note that $E^{(q)}$ is functorial for flat morphisms of schemes. 

These constructions can be compared.  Cycles on $\Delta^\bullet \times X \times \bb{A}^q$ that are quasi-finite over $\Delta^\bullet \times X$ are cycles of codimension at least $q$ intersecting all faces of $\Delta^\bullet$ in codimension at least $q$.  Levine \cite[Thm.~3.3.5]{MLChow} has shown that, if $k$ is infinite, the projection $\bb{A}^1 \times X \to X$ induces a weak equivalence $E^{(q)}(\bb{A}^1 \times X ) \cong E^{(q)}(X)$ for any $X \in {\bf{Sm}}/k$.  Thus there is a canonical map $\alpha_X: E^Q(X) \into E^{(q)}(X \times \bb{A}^q) \cong E^{(q)}(X)$.  Our main result is that $\alpha_X$ is a weak equivalence for a broad class of theories $E$, but with some restrictions on $X$.

\begin{theorem}[$\ref{main comparison affine}, \ref{main comparison projective}$] Let $k$ be an infinite field.  Let $E \in {\bf{Spt}}(k)$ be a presheaf of spectra satisfying Nisnevich excision (see Section $\ref{coh thy defn}$).  Suppose $X$ is a smooth equidimensional $k$-scheme that is either affine or projective.  Then the canonical map $\alpha_X : E^Q(X) \to E^{(q)}(X)$ is a weak equivalence. \end{theorem}

For $E \in {\bf{Spt}}(k)$ quasi-fibrant and $X \in {\bf{Sm}} /k$, Levine conjectured $\alpha_X$ is a weak equivalence \cite[Conj.~10.1]{ST}.  Levine also observes that, using the identification of the $q$th slice $s_qE$ of $E$ with the cofiber of $E^{(q+1)} \to E^{(q)}$ \cite[Thm.~1.1]{MLHtyCon}, the conjecture would give an explicit functorial geometric model for the slice tower of an $S^1$-spectrum.  Our result gives a functorial geometric model for the slice tower on affine and projective schemes, namely the tower $\cdots E^{Q, q+1} \to E^{Q,q} \cdots$ determined by any sequence of coordinate embeddings $\cdots \bb{A}^q \into \bb{A}^{q+1} \cdots$.  We note that a functorial model for $E^{(q)}$ is obtained in \cite[Sect.~7]{MLChow} using a moving argument adapted to a nonflat morphism, together with the Dwyer-Kan construction.  See Remark $\ref{ML funct}$ for further discussion.

Our result has a structural interpretation in terms of the injective Nisnevich-local $\A$-model structure on ${\bf{Spt}}(k)$.  Given a fibrant $E \in {\bf{Spt}}(k)$ for this model structure, a basic problem is to find a fibrant model for $E^{(q)}$.  Regarding the $\A$-part of the model structure, a useful feature of the algebraic singular complex is that it ``creates" homotopy invariance: if $\cal{F} \in {\bf{Spt}}(k)$ is a presheaf of spectra, then the assignment $X \mapsto |n \mapsto \cal{F}(\Delta^n \times X)| =: \cal{F}(\Delta^\bullet \times X)$ is a homotopy invariant presheaf of spectra on ${\bf{Sm}} /k$ \cite[Prop.~7.2]{FS}.  In particular, for any $E \in {\bf{Spt}}(k)$, the presheaf $E^Q \in {\bf{Spt}}(k)$ is homotopy invariant.  Since $E^{(q)}$ satisfies Nisnevich excision, in some sense the main question of this paper is whether the operation $E \mapsto E^Q$ preserves Nisnevich-fibrant objects (see Remark $\ref{Nis exc remark}$ for further discussion).  A fibrant presheaf can be used to compute mapping spaces in the homotopy category $\cal{SH}_{S^1}(k)$, e.g., cohomology groups.  In view of the fundamental role of the operation $E \mapsto E^{(q)}$, it is interesting to obtain an explicit fibrant model for $E^{(q)}$ in ${\bf{Spt}}(k)$.

The map $\alpha_X$ has been studied for particular theories $E$.  The comparison map is a weak equivalence for sheaves of equidimensional cycles, where the Suslin complex $C_*(-)$ plays the role of the algebraic singular complex.  A case of Friedlander-Voevodsky duality \cite[Thms.~7.1,7.4]{BCC} implies $z_{equi}(X, \bb{A}^q,0) \to z_{equi}(X \times \bb{A}^q, \dim(X))$ induces a quasi-isomorphism of Suslin complexes.  We have also Suslin's comparison of equidimensional cycles with Bloch cycles \cite[Thm.~3.2]{Sus}, namely a quasi-isomorphism $C_*(z^q_{equi}(X)) \to z^{(q)}(\Delta^\bullet \times X)$. Thus the map $z^Q(\Delta^\bullet \times X \times \bb{A}^q) \to z^{(q)}(\Delta^\bullet \times X \times \bb{A}^q) \cong z^{(q)}(\Delta^\bullet \times X)$ is a quasi-isomorphism, at least if $k$ admits resolution of singularities.  Our result can then be viewed as a form of duality for a rather general theory $E$, with $E^Q$ playing the role of $E$-cohomology, and $E^{(q)}$ that of $E$-homology.

For algebraic $K$-theory, the comparison map was used by Friedlander-Suslin in the construction of a spectral sequence relating motivic cohomology to algebraic $K$-theory \cite{FS}; this is a motivic analogue of the spectral sequence of Atiyah-Hirzebruch relating the singular cohomology of a topological space to its topological $K$-theory.  More precisely, the map $\alpha_X$ was shown to be a weak equivalence for $X$ the spectrum of a field or a semilocal ring.  See Remark $\ref{FS comparison}$ for further discussion.  Algebraic $K$-theory satisfies Nisnevich excision by a result of Thomason-Trobaugh \cite[Prop.~3.19]{TT}, hence our result generalizes the comparison results of Friedlander-Suslin.

Our techniques are primarily geometric in nature.  First we show that the map $\alpha_{\bb{A}^d}$ is a weak equivalence; this result is valid for any presheaf $E \in {\bf{Spt}}(k)$ and the proof follows closely \cite[Sec.~8]{FS}.  The main ingredient allowing us to move simplicial closed subsets on affine space is the moving lemma of Suslin \cite[Thm.~1.1]{Sus}, which moves a given family to an equidimensional family.  Then, assuming $E$ satisfies Nisnevich excision, we obtain the result for a general smooth affine $X$ by choosing a finite morphism $f: X \to \bb{A}^d$ adapted to a particular collection of supports on $X$.  An important technical device is the definition of a sequence of support conditions refining the inclusion of the quasi-finite supports into the supports intersecting all the faces properly.  We prove the resulting tower of spectra has contractible cofibers for $\bb{A}^d$, and that the discrepancy is sufficiently small that cycles on a smooth affine $X$ can be moved to a smaller support condition in the filtration.  The basic principle goes back at least to the classical form of Chow's moving lemma, the core of which is the idea that the cycle $f^*(f_*(Z)) -Z$ on $X$ should have better intersection properties than $Z$ itself.  Thus we adopt the viewpoint of Levine \cite{MLChow}.

The argument for a smooth projective $X$ is similar.  We analyze the proof \cite[Thm.~1.1]{Sus}, carrying along a hyperplane.  The key geometric input (Theorem $\ref{projective suslin}$) is that the moving morphisms for affine space can be extended to projective space, without disturbing the support condition at infinity.  This is not a direct generalization of Suslin's moving lemma to the projective case, but after a careful choice of hyperplane section, some excision arguments allow us to more-or-less appeal to the study of $\alpha_X$ for $X$ affine.  The basic difficulty presented by the general smooth quasi-projective case is that our support conditions are not preserved by taking closures; see Remark $\ref{qproj failure}$ for further discussion.

\smallskip

\textbf{Acknowledgments.} This problem was suggested to the author by Marc Levine.  The author wishes to thank him for a series of very helpful conversations during which many of the basic ideas presented here emerged.  This work was begun while the author was a wissenschaftlicher Mitarbeiter at the Universit\"at Duisburg-Essen.

\section{Support conditions and cohomology theories} 

\subsection{Support conditions} \label{sec supports}

Let $\Delta^n $ denote the algebraic $n$-simplex $\spec k[t_0, \ldots , t_n] / (t_0 + \cdots + t_n = 1)$.  Letting $n$ vary and using the coface and codegeneracy maps (see, e.g., \cite[p.~268]{Blhigher}), we obtain the cosimplicial object $\Delta^\bullet$ in ${\bf{Sm}}/k$.  The goal of this paper is to compare various \textit{support conditions} $S(X)$ on cohomology theories evaluated on the cosimplicial schemes $\Delta^\bullet \times X \times \bb{A}^q$ for $X \in {\bf{Sm}} / k$.

Suppose $S_n(X)$ is a collection of closed subschemes of $\Delta^n \times X \times \bb{A}^q$ that is closed under finite unions and under taking closed subschemes.  Then the collections $\{ S_n(X) \}_n$ comprise a support condition if they form a simplicial subset of the simplicial set whose $n$-simplices are the closed subschemes of $\Delta^n \times X \times \bb{A}^q$ \cite[Defn.~2.3.1]{MLChow}.  In this context we often refer to the closed subschemes themselves as supports.  The support condition $S(X)$ is of codimension $\geq q$ if, for all $n$, every $Z \in S_n(X)$ is of codimension $\geq q$ on $\Delta^n \times X \times \bb{A}^q$.  The collections $Q$ and $S^{(q)}$ defined in Section $\ref{sec intro}$ are support conditions of codimension $\geq q$ on $\Delta^\bullet \times X \times \bb{A}^q$.  Support conditions are partially ordered by inclusion.  In particular, a collection of subschemes $\{Z_i^n \into \Delta^n \times X \times \bb{A}^q\}_n$ is said to generate a support condition $S(X)$ if $S$ is the smallest support condition containing all the $Z_i^n$s.  We use the notation $Z' \into Z \subset S(X)$ to mean every member of $Z'$ is contained in some member of $Z$.

Let $X$ be a smooth equidimensional $k$-scheme.  For $e \in \bb{Z}_{\geq 0}$, let $S^{(q), e}_n(X) \subset S^{(q)}_n(X)$ denote the set of codimension $\geq q$ subschemes $Z \into \Delta^n \times X \times \bb{A}^q$ having the additional property that the maximal fiber dimension of the canonical morphism $Z \to \Delta^n \times X$ is $e$.  The condition on fiber dimension is preserved by change of the base $\Delta^n \times X$, in particular by the coface and codegeneracy maps, so the sets $S^{(q),e}_n(X)$ define a support condition on $\Delta^\bullet \times X \times \bb{A}^q$.

To use geometric arguments we will need to further refine the filtration
\begin{equation} \label{dim support tower} Q(X) = S^{(q),0}(X) \subset S^{(q),1}(X) \subset \cdots \subset S^{(q),q-1}(X) \subset S^{(q),q}(X) = S^{(q)}(X). \end{equation}

To this end, let $S^{(q), e, f}_n(X) \subset S^{(q), e}_n(X)$ denote the set of subschemes $Z$ such that the image of the locus of $e$-dimensional fibers has codimension $\geq f$ in $\Delta^n \times X$, and continues to have codimension $\geq f$ on each face $\Delta^m \times X \into \Delta^n \times X$.  Since the codimension of the image of the locus of maximal dimension fibers is preserved by flat base change, this too defines a support condition.  Now suppose $X$ is of dimension $d$.  Then if $e$ and $n$ are fixed, by varying $f$ we obtain an increasing chain of support conditions:
\begin{equation} \label{support tower} S^{(q), e, d+n}_n(X) \subset S^{(q), e, d+n-1}_n(X) \subset \cdots \subset S^{(q), e, 0}_n(X) = S^{(q),e}_n(X). \end{equation}

Here we collect some elementary properties of our support conditions.

\begin{properties}  \label{elem support properties}
 \begin{enumerate}
\item (finite covariant functoriality) The support conditions $S^{(q),e,f}(-)$ are compatible with pushforward by a finite (flat) morphism in ${\bf{Sm}} /k$.  For such a morphism $f : X \to Y$, we have $f_*(S^{(q),e,f}_n(X)) \subseteq S^{(q),e,f}_n(Y)$. \label{elem support properties 2}
\item (flat contravariant functoriality) The support conditions $S^{(q),e,f}(-)$ are compatible with pullback by a flat morphism in ${\bf{Sm}} /k$.  For such a morphism $f : X \to Y$, we have $f^* (S^{(q),e,f}_n(Y)) \subseteq S^{(q),e,f}_n(X)$. \label{elem support properties 3}
\item (openness in families) Let $B \in {\bf{Sm}}/k$ be a base scheme, and let $Z \in S^{(q)}_n(X \times B)$ be a codimension $q$ cycle on $\Delta^n \times X \times B \times \bb{A}^q$.  Suppose for $b \in B$, the fiber $Z_b$ belongs to $S^{(q),e,f}_n(X \times \{b \})$.  Then there exists an open neighborhood $b \in B' \subseteq B$ such that $Z_{b'} \in S^{(q),e,f}_n(X \times \{b' \})$ for all $b' \in B'$. \label{elem support properties 4}
   \end{enumerate}
\end{properties}

\subsection{Cohomology theories}

Here we explain some notation from the introduction and some constructions, and discuss axioms for the cohomology theories $E$ we will consider.

A spectrum $E$ is a sequence of pointed simplicial sets $E_0, E_1, E_2, \ldots$ together with bonding maps $S^1 \wedge E_i \to E_{i+1}$.  A morphism of spectra is a sequence of pointed morphisms of simplicial sets compatible with the bonding maps, and ${\bf{Spt}}$ denotes the category of spectra.  A presheaf of spectra on ${\bf{Sm}}/k$ is an additive functor $E : {({\bf{Sm}} / k)}^{op} \to {\bf{Spt}}$, i.e., a functor with the property that $E(X \coprod Y) \to E(X) \oplus E(Y)$ is a weak equivalence for all $X,Y \in {\bf{Sm}}/k$.

Let $E$ be a presheaf of spectra on ${\bf{Sm}} /k$.  We consider the following conditions on $E$.
\begin{enumerate} \label{coh thy defn}
\item (Nisnevich excision) Let $f : X \to Y$ be an \'etale morphism in ${\bf{Sm}} / k$, and suppose $W \into Y$ is a closed subscheme such that the induced map $f^{-1}(W) \to W$ is an isomorphism of reduced schemes.  Then $f$ induces a weak equivalence $E^W(Y) \xrightarrow{\sim} E^{f^{-1}(W)}(X)$.
\item (homotopy invariance) Let $X \in {\bf{Sm}} / k$ and let $p : X \times \A \to X$ denote the projection.  Then $p$ induces a weak equivalence $E(X) \xrightarrow{\sim} E(X \times \A)$.
\end{enumerate}

\begin{remark} For some time we will consider arbitrary $E \in {\bf{Spt}}(k)$, but we need to impose Nisnevich excision starting in Section $\ref{general affine}$ to obtain the comparison result for affine varieties.  Our results do not require the hypothesis of homotopy invariance.  The relevance of this condition is the $\A$-part of the model structure on presheaves of spectra.  \end{remark}

For $Z \into X$ a closed subscheme of the smooth $k$-scheme $X$ and $E \in {\bf{Spt}} (k)$, we define the spectrum on $X$ with supports in $Z$, denoted $E^Z(X)$, to be the homotopy fiber of the morphism $E(X) \to E(X - Z)$.  This is the spectrum whose homotopy groups fit into a long exact sequence with those of $E(X)$ and $E(X - Z)$.  The spectrum $E^Z(X)$ depends only on the support of $Z$ and not on its scheme structure.

If $S(X)$ is a support condition and $E \in {\bf{Spt}}(k)$ is a presheaf of spectra on ${\bf{Sm}} / k$, then $n \mapsto E^{S_n}(X) := E^{S_n(X)}(\Delta^n \times X \times \bb{A}^q) := \hocolim_{Z \in S_n(X)}E^Z(\Delta^n \times X \times \bb{A}^q)$ forms a simplicial spectrum.  Then we may form the total spectrum $E^S(X) := |n \mapsto E^{S_n}(X)|$.  If $S_1(X)$ and $S_2(X)$ are support conditions and $S_1(X) \subseteq S_2(X)$, there is a canonical morphism $E^{S_1}(X) \to E^{S_2}(X)$.

We will write $E^{(q)}(X)$ for $E^{S^{(q)}(X)}(X)$, and we set $E^{(q),e}(X) := E^{S^{(q),e}(X)}(X)$.  The filtration $\ref{dim support tower}$ induces a sequence of canonical morphisms 
$$E^Q(X) \to E^{(q),1}(X) \to \cdots \to E^{(q), q-1}(X) \to  E^{(q), q}(X) = E^{(q)}(X)$$
refining the comparison map $\alpha_X : E^Q(X) \to E^{(q)}(X)$.  

By imposing the condition $n \leq N$, we obtain truncated spectra $E^{(q),e}(X)_{\leq N}$ and morphisms of spectra $E^{(q),e,f}(X)_{\leq N} \to E^{(q),e,f-1}(X)_{\leq N}$ (with the evident notation) fitting into a tower:
$$ \cdots \to E^{(q),e-1,0}(X)_{\leq N} \to E^{(q),e,d+N}(X)_{\leq N} \to E^{(q),e,d+N-1}(X)_{\leq N} \to \cdots. $$
refining the canonical map $E^Q(X)_{\leq N} \to E^{(q)}(X)_{\leq N}.$  We use the filtration as follows.  We will show that for any $N$, the morphism $E^{(q),e-1,0}(X)_{\leq N} \to E^{(q),e,d+N}(X)_{\leq N}$ and all of the morphisms $E^{(q),e,f}(X)_{\leq N} \to E^{(q),e,f-1}(X)_{\leq N}$ are weak equivalences.  Since the $r$th homotopy group of a simplicial spectrum depends only on some truncation (in simplicial degree depending on $r$), it follows that $E^{(q),e-1}(X) \to E^{(q),e}(X)$ is a weak equivalence for every $e$.  Hence also $\alpha_X : E^Q(X) \to E^{(q)}(X)$ is a weak equivalence.

\begin{remarks} \label{coh remarks} \subtheorem Throughout we assume $k$ is infinite to guarantee the existence of morphisms to affine (and projective) space with certain properties.  Levine's results \cite{MLChow}, e.g., the homotopy invariance of $E^{(q)}$ cited in the introduction, can be obtained for $k$ a finite field by imposing a Galois invariance hypothesis \cite[Axiom 4.1.3]{MLChow}.  It seems reasonable to assume the results obtained here can be similarly extended in the presence of such a hypothesis.
\subtheorem We work in roughly the same generality as \cite{MLChow}; in particular we do not require $E$ to be homotopy invariant.  The localization axiom (4.1.1) of \cite{MLChow} is automatic in our setting since we consider presheaves of spectra on ${\bf{Sm}} /k$ rather than presheaves on some category of pairs in ${\bf{Sm}} /k$.  See \cite[Defn.~2.1.1; 9.1]{MLChow}.  For completeness we simply record that for a sequence of closed immersions $Z' \into Z \into X$ with $X \in {\bf{Sm}} /k$, we have a fibration sequence $E^{Z'}(X) \to E^Z(X) \to E^{Z \setminus Z'}(X \setminus Z')$.
\subtheorem As a particular case of the Nisnevich excision axiom (Zariski excision), if $j : U \subset X$ is an open immersion in ${\bf{Sm}} /k$ and $Z \into X$ is a closed subscheme which happens to be contained in $U$, then $j$ induces a weak equivalence $E^Z(X) \xrightarrow{\sim} E^Z(U)$.  This weak equivalence yields the following excision result: let $f : X \to Y$ be a finite, generically \'etale morphism in ${\bf{Sm}} / k$ with ramification locus $R_f \into X$, let $Z' \into Z \into X$ be closed subschemes such that $f$ induces an isomorphism $Z \setminus Z' \cong f(Z \setminus Z')$, and suppose $Z' \supseteq R_f \cap Z$.  Then $f$ induces a weak equivalence $E^{f(Z) \setminus f(Z')}(Y \setminus f(Z')) \xrightarrow{\sim} E^{Z \setminus Z'}(X \setminus Z').$ 
\end{remarks}

\section{Moving supports on affine varieties}

\subsection{Moving via endomorphisms of $\Delta^\bullet \times \bb{A}^d \times \bb{A}^q$}

It is a theorem of Friedlander-Suslin \cite[Thm.~8.6]{FS} that for $E = K$, the spectrum of (Thomason) $K$-theory, the canonical maps $E^Q(\Delta^\bullet \times \bb{A}^q) \hookrightarrow E^{(q)}(\Delta^\bullet \times \bb{A}^q ) \hookleftarrow E^{(q)}(\Delta^\bullet)$ are weak equivalences.  In this section, following very closely the argument in \cite[Sec.~8]{FS}, we enhance this result in three ways:

\begin{enumerate}
\item replace $K$ with an arbitrary presheaf of spectra;
\item replace $\Delta^\bullet \times \bb{A}^q$ with $\Delta^\bullet \times \bb{A}^d \times \bb{A}^q$; and
\item replace $S^{(q)}(\bb{A}^d)$ with certain intermediate support conditions $Q(\bb{A}^d) \subseteq S(\bb{A}^d) \subseteq S^{(q)}(\bb{A}^d)$, especially those discussed in Section $\ref{sec supports}$.
\end{enumerate}

In this section we consider support conditions $S(X)$ of codimension $\geq q$ satisfying $Q(X) \subseteq S(X) \subseteq S^{(q)}(X)$, and $E \in {\bf{Spt}}(k)$ arbitrary.  We say $Z := \{Z_i^n \} \subset S_{\leq N}(X)$ is a finite subfamily of supports if the set $Z$ is finite.  This means there exists $N \in \bb{Z}_{\geq 0}$ and for each $0 \leq n \leq N$ there is specified a finite subset $\{Z_i^n \} \subset S_n(X)$.  By convention we assume the set $Z$ is closed under intersecting with faces $\Delta^n \into \Delta^N$, hence we can form the spectrum $E^Z(X)$, the total spectrum of the truncated simplicial spectrum $n \mapsto E^{ \{Z_i^n \} }(\Delta^n \times X \times \bb{A}^q)$.

\begin{definition} A \textit{pseudo-endomorphism} $\phi_\bullet$ of $\Delta^\bullet \times \bb{A}^d \times \bb{A}^q$ is a family of $\bb{A}^q$-morphisms $\phi_\bullet : \Delta^\bullet \times \bb{A}^d \times \bb{A}^q \to \Delta^\bullet \times \bb{A}^d \times  \bb{A}^q$ such that for every strictly increasing $\theta :[m] \to [n]$, the diagram
$$\xym{\Delta^m \times \bb{A}^d \times \bb{A}^q \ar[r]^-{\phi_m} \ar[d]^{\theta_* \times 1 \times 1} & \Delta^m \times \bb{A}^d \times \bb{A}^q \ar[d]^-{\theta_* \times 1 \times 1} \\
\Delta^n \times \bb{A}^d \times \bb{A}^q \ar[r]^-{\phi_n} & \Delta^n \times \bb{A}^d \times \bb{A}^q \\}$$
commutes.  

A \textit{homotopy} between $\phi_\bullet$ and the identity is a pseudo-endomorphism $\Phi_\bullet$ of $\Delta^\bullet  \times \A \times \bb{A}^d \times \bb{A}^q$ such that $i_0 \circ \phi_\bullet = \Phi_\bullet \circ i_0$ and $i_1 = \Phi_\bullet \circ i_1$ as morphisms $\Delta^\bullet \times \bb{A}^d \times \bb{A}^q \to \Delta^\bullet \times \A \times \bb{A}^d \times \bb{A}^q$.  \end{definition}

For a pseudo-endomorphism $\phi_\bullet$ of $\Delta^\bullet \times \bb{A}^d \times \bb{A}^q$ and a codimension $q$ support condition $S(\bb{A}^d)$, let $\leftsub{\phi}{S}_n(\bb{A}^d)$ denote the supports $Z \in S_n(\bb{A}^d)$ such that $\phi_n^{-1}(Z) \in S_n(\bb{A}^d)$.  This is again a support condition, and it satisfies $\leftsub{\phi}{S}(\bb{A}^d) \subset S(\bb{A}^d)$.  For each $n$ we have a map of spectra $\phi_n^* : E^{\leftsub{\phi}{S}_n}(\Delta^n \times \bb{A}^d \times \bb{A}^q) \to E^{S_n}(\Delta^n \times \bb{A}^d \times \bb{A}^q)$.  Since $\phi_\bullet$ is a pseudo-endomorphism these induce a map on Segal realizations
$$\phi^* : || n \mapsto | E^{\leftsub{\phi}{S}_n}(\Delta^n \times \bb{A}^d \times \bb{A}^q)| || \to || n \mapsto |E^{S_n}(\Delta^n \times \bb{A}^d \times \bb{A}^q)| ||.$$

From now on, we use the notation $|| E^S (\bb{A}^d) ||   := || n \mapsto | E^{S_n} (\Delta^n \times \bb{A}^d \times \bb{A}^q ) | ||$.

\begin{proposition}\label{FS 8.2} For any finite subfamily of supports $Z \subset S_{\leq N}(\bb{A}^d)$, there exists a pseudo-endomorphism $\phi_\bullet : \Delta^\bullet \times \bb{A}^d \times \bb{A}^q  \to \Delta^\bullet \times \bb{A}^d \times \bb{A}^q$ such that $\phi_n^{-1}(Z_i^n) \in Q_n( \bb{A}^d)$ for all $Z_i^n \in Z$.  \end{proposition}
\begin{proof} By combining Suslin's moving lemma and induction on the simplicial degree $n$, \cite[Prop.~8.2]{FS} shows that a finite family of supports in $S^{(q)}_n(\spec k)$ (i.e., on $\Delta^n \times \bb{A}^q$) can be moved via a pseudo-endomorphism to a family of $\Delta^n$-quasi-finite supports.  The proof carries over to our setting.  Suslin's moving lemma \cite[Thm.~1.1]{Sus} asserts that given an affine $k$-scheme $S$, an effective divisor $D \into \bb{A}^n$, a closed subscheme $Z$ of $\bb{A}^n \times S$, and an $S$-morphism $\phi_D : D \times S \to \bb{A}^n \times S$, there exists an an $S$-endomorphism $\phi$ of $\bb{A}^n \times S$ transporting $Z$ to a closed subscheme having the correct fiber dimension over $\bb{A}^n  \setminus D$, and agreeing with $\phi_D$ on $D \times S$.

First we find, for each vertex $\Delta^0 \into \Delta^n$, a morphism $\phi_0 : \Delta^0 \times \bb{A}^d \times \bb{A}^q \to \Delta^0 \times \bb{A}^d \times \bb{A}^q$ moving $Z^0$ to a $\Delta^0 \times \bb{A}^d$-quasi-finite support.  (Here $Z^m := Z \cap (\Delta^m \times \bb{A}^d \times \bb{A}^q)$ denotes a face of $Z$.)  Suslin's lemma applies so long as $Z^0$ has codimension $q$, so we can replace the condition $S^{(q)}$ in \cite[Prop.~8.2]{FS} with a smaller support condition $S(\bb{A}^d) \subseteq S^{(q)}(\bb{A}^d)$.  

Then we find, for each face $\Delta^1 \into \Delta^n$, an endomorphism $\phi_1$ of $\Delta^1 \times \bb{A}^d \times \bb{A}^q$ moving $Z^1$ (at least away from the vertices) and agreeing with $\phi_0$ on $(\cup \Delta^0) \times \bb{A}^d \times \bb{A}^q$.  Because $\phi_0$ moves $Z^0$, we conclude $\phi_1^{-1}(Z^1)$ is quasi-finite over the vertices $(\cup \Delta^0) \times \bb{A}^d$ as well.  Furthermore, because the $\phi_1$'s agree on the vertices, they glue to an endomorphism of $(\cup \Delta^1) \times \bb{A}^d \times \bb{A}^q$ moving $Z |_{\cup \Delta^1}$.  We continue one simplicial degree at a time.  In the last stage, we may replace the divisor $V(t_0 \cdots t_n) \into \Delta^n$ with the divisor $V(t_0 \cdots t_n) \into \Delta^n \times \bb{A}^d$ in the proof of \cite[Prop.~8.2]{FS}. \end{proof}

\begin{corollary}\label{FS 8.3} For any compact subset $K \subset || E^S (\bb{A}^d)||$, there exists a pseudo-endomorphism $\phi_\bullet$ such that $K \subset || E^{\leftsub{\phi}{S}}(\bb{A}^d)||$ and $\phi^*(K) \subset || E^Q(\bb{A}^d)||$.
\end{corollary}

\begin{notation} \label{simplicial interval} Let $I_\bullet$ denote the simplicial set corresponding to the poset $\{ 0 < 1 \}$.  We denote by $0_\bullet \subset I_\bullet$ the constant simplicial subset at $0$: for every $n$ its unique $n$-simplex is the sequence $0 \leq \cdots \leq 0$.  Similarly we have $1_\bullet \subset I_\bullet$.  For $j \in I_n$ let $f_j : \Delta^n \to \Delta^n \times \A$ denote the (unique) linear morphism sending the $k^{th}$ vertex $v_k$ to $v_k \times j_k$.  We use the simplicial set $I_\bullet$ to construct homotopies; the geometric realization $|I_\bullet|$ is the unit interval $I$, and via this realization, $|0_\bullet|, |1_\bullet|$ correspond to $0,1$ respectively. \end{notation}

Following \cite{FS} we have the support condition $\leftsub{\Phi}{S}$ on $\Delta^\bullet \times \A \times \bb{A}^d \times \bb{A}^q$ consisting of subschemes in $S(\A \times \bb{A}^d)$ belonging to $S$ on $\Delta^\bullet \times \bb{A}^d \times \bb{A}^q$ after pullback by any morphism of the form $\Phi_n \circ (f_j \times 1 \times 1)$.  Similarly we have $\leftsub{\Phi}{S}$ on $\Delta^\bullet \times \bb{A}^d \times \bb{A}^q$, those subschemes that lie in $\leftsub{\Phi}{S}(\A \times \bb{A}^d)$ upon pullback via the projection $p \colon \Delta^\bullet \times \A \times \bb{A}^d \times \bb{A}^q \to \Delta^\bullet \times \bb{A}^d \times \bb{A}^q$.

\begin{proposition}\label{FS 8.4} Let $S$ be a support condition on $\Delta^\bullet \times \bb{A}^d \times \bb{A}^q$ of the form $S^{(q), e, f}(\bb{A}^d)$.  Then for any finite subfamily of supports $Z \subset \leftsub{\phi}{S}_{\leq N}(\bb{A}^d)$, there exists a homotopy $\Phi_\bullet : \Delta^\bullet \times \A \times \bb{A}^d \times \bb{A}^q  \to \Delta^\bullet \times \A \times \bb{A}^d \times \bb{A}^q$ between $\phi_\bullet$ and the identity such that $Z_i^n \in \leftsub{\Phi}{S}_n(\bb{A}^d)$ for all $Z_i^n \in Z$.  \end{proposition}

\begin{proof} Suslin's moving lemma and induction on the index of the pseudo-endomorphism $\Phi_\bullet$ allow one to construct a candidate $\Phi_\bullet$ as in \cite[8.4]{FS}.  Since supports of the form $p^{-1}(S^{(q)}(\bb{A}^d))$ on $\Delta^\bullet \times \A \times \bb{A}^d \times \bb{A}^q$ have codimension $\geq q$ after pullback by $\Phi_n \circ (f_j \times 1 \times 1)$, the same is true of $p^{-1}(S(\bb{A}^d))$ for a smaller support condition $S$.  So it remains to check that the indices $e$ and $f$ are preserved.

The candidate $\Phi_\bullet$ has the property that 
$\Phi_n^{-1}( p^{-1}(Z_i^n))$ is quasi-finite over $\Delta^n \times (\A \setminus \{0,1 \} ) \times \bb{A}^d$.  Thus for arbitrary $j \in I_n$, the non quasi-finite locus of
$${(\Phi_n \circ (f_j \times 1 \times 1))}^{-1}( p^{-1}(Z_i^n))$$

is contained in the non quasi-finite locus of
$${(\Phi_n \circ (f_{0_n} \times 1 \times 1))}^{-1}( p^{-1}(Z_i^n)) \cup {(\Phi_n \circ (f_{1_n} \times 1 \times 1))}^{-1}(p^{-1}(Z_i^n)).$$

Over $\Delta^n \times \{ 0 \} \times \bb{A}^d$ we have the cycle $\phi_n^{-1}(Z_i^n)$, and over $\Delta^n \times \{ 1 \} \times \bb{A}^d$ we simply recover the cycle $Z_i^n$.   Thus by hypothesis we have ${(\Phi_n \circ (f_j \times 1 \times 1))}^{-1}( p^{-1}(Z_i^n)) \in S_n(\bb{A}^d)$ for $j = 0_n, 1_n$.  By the previous paragraph, the indices $e$ and $f$ for arbitrary $j \in I_n$ can only improve. \end{proof}

\begin{remark} If $\phi_\bullet$ satisfies the conclusion of Proposition $\ref{FS 8.2}$, then in the proof of Proposition $\ref{FS 8.4}$ we would only have to consider incidences of the form
$${(\Phi_n \circ (f_{j} \times 1 \times 1))}^{-1}( p^{-1}(Z_i^n)) \cap {(\Phi_n \circ (f_{1_n} \times 1 \times 1))}^{-1}(p^{-1}(Z_i^n)).$$  \end{remark}

\begin{proposition}\label{FS 8.5} Suppose $S$ is a support condition on $\Delta^\bullet \times \bb{A}^d \times \bb{A}^q$ of the form $S^{(q), e, f}(\bb{A}^d)$.  Let $\phi_\bullet$ be a pseudo-endomorphism of the cosimplicial scheme $\Delta^\bullet \times \bb{A}^d \times \bb{A}^q$.  Then the morphism
$$ || E^{\leftsub{\phi}{S}}(\bb{A}^d) || \xrightarrow{\phi^*} || E^S (\bb{A}^d ) ||$$
is weakly homotopic to the canonical inclusion map. \end{proposition}
\begin{proof} Having established Proposition $\ref{FS 8.4}$, the proof of \cite[Prop.~8.5]{FS} applies. \end{proof}

\begin{theorem}\label{FS 8.6} Suppose $S$ is a support condition on $\Delta^\bullet \times \bb{A}^d \times \bb{A}^q$ of the form $S^{(q), e, f}(\bb{A}^d)$.  Then the canonical map $E^Q(\bb{A}^d ) \to E^S( \bb{A}^d )$ is a weak equivalence.  In particular $\alpha_{\bb{A}^d} : E^Q(\bb{A}^d) \to E^{(q)}(\bb{A}^d)$ is a weak equivalence.  \end{theorem}
\begin{proof} Corollary $\ref{FS 8.3}$ and Proposition $\ref{FS 8.5}$ imply $|| E^Q(\bb{A}^d ) || \to || E^S (\bb{A}^d)||$ is a weak equivalence.  Every codegeneracy morphism $\Delta^n \to \Delta^{n-i}$ is a linear projection with a section.  Therefore, for every term $E_i$ of $E = (E_0, E_1, \cdots )$, the morphism $|E_i^S(\Delta^{n-i} \times \bb{A}^d \times \bb{A}^q)| \to |E_i^S(\Delta^n \times \bb{A}^d \times \bb{A}^q)|$, being the geometric realization of a monomorphism of simplicial sets, is a closed cofibration.  Thus each simplicial space $|(n \mapsto |E_i(\Delta^n \times \bb{A}^d \times \bb{A}^q)|)|$ is good in the sense of \cite[Defn.~A.4]{SegalCats}.  Therefore we may apply Segal's theorem \cite[App.~A]{SegalCats}, which says the canonical map from the Segal realization of a good simplicial space to its usual realization is a weak equivalence.  \end{proof}

\subsection{Controlled moving on $\Delta^\bullet \times \bb{A}^d \times \bb{A}^q$} \label{controlled moving}

Suppose $Z \subset S_{\leq N}^{(q)}(\bb{A}^d)$ is a finite subfamily of supports.  We have shown the canonical map $E^Z(\bb{A}^d) \to E^{(q)}(\bb{A}^d)$ fits into a homotopy commutative diagram (in which the unlabeled arrows are canonical):
\begin{equation} \label{supports diagram} \xym{E^Z(\bb{A}^d) \ar[r] \ar[d]^{\phi^*} & E^{(q)}(\bb{A}^d) \\
E^{\phi^{-1}(Z)}(\bb{A}^d) \ar[r] & E^Q(\bb{A}^d) \ar[u] \\ } \end{equation}

The homotopy commutativity means $\phi^*$ is homotopic to the inclusion, and the image of $\phi^*$ factors through the smaller space $E^Q(\bb{A}^d) \subset E^{(q)}(\bb{A}^d)$.  Since the spaces on the left hand side involve finitely many supports and the interval is compact, the spaces on the right hand side may be replaced by spectra involving only finitely many supports.  We will need a rather precise description of this finite collection.  So for $Z \subset S_n^{(q)}(\bb{A}^d)$, we denote by ${F^*Z} \subset S_n^{(q)}(\bb{A}^d)$ the support condition generated by the subschemes ${(\Phi_n \circ (f_j \times 1 \times 1))}^{-1} (p^{-1}(Z)) \into \Delta^n \times \bb{A}^d \times \bb{A}^q$ for $j \in I_n$.  This support condition is the subject of Proposition $\ref{FS 8.4}$.

Informally, the ``worst" support in $F^*Z$, measured by the failure of quasi-finiteness, is $Z$ itself.  Then we have supports that share a facet $\Delta^{n-1} \into \Delta^n$ with $Z$, but are quasi-finite over $(\Delta^n \setminus \Delta^{n-1}) \times \bb{A}^d$.  Then we have supports that coincide with $Z$ along a codimension 2 face of $\Delta^n$, but are otherwise quasi-finite; and so on, until we reach $\phi_n^{-1}(Z)$, which is quasi-finite over all of $\Delta^n \times \bb{A}^d$.  We will use $R^*Z$ to denote the support condition where we use all of the $f_j$'s except $f_{1_n}$.  So we have $R^*Z = \oline{F^*Z \setminus Z}$ and $F^*Z = Z \cup R^*Z$.  

\begin{remark} \label{move preserves tower} If $Z \subset S^{(q),e,f}_n(\bb{A}^d)$ and $\phi_\bullet$ is a pseudo-endomorphism such that $\phi_n^{-1}(Z)$ is quasi-finite over $\Delta^n \times \bb{A}^d$, then there exists a homotopy $\Phi_\bullet$ between $\phi_\bullet$ and the identity such that $F^* Z \subset S^{(q),e,f}_n(\bb{A}^d)$.  This is a restatement of Proposition $\ref{FS 8.4}$.  \end{remark}

Proposition $\ref{no new supports}$ below implies that, in the diagram $\ref{supports diagram}$, $E^{(q)}(\bb{A}^d)$ may be replaced by the spectrum $E^{F^*Z}(\bb{A}^d)$.  In this sense no ``new" supports in $S^{(q)}(\bb{A}^d) \setminus Q(\bb{A}^d)$ are encountered in the move; this is implicit in the proof of Proposition $\ref{FS 8.4}$.  To state the precise result, we use the following notation.  If $Z, W \into \Delta^\bullet \times X \times \bb{A}^q$ are closed subschemes, then we denote by $E^{Z \setminus Z \cap W}(X \setminus W)$ the theory with supports on $Z \setminus Z \cap W$, evaluated on $\Delta^\bullet \times X \times \bb{A}^q \setminus W$.

\begin{proposition} \label{no new supports} Let $Z \subset S_{\leq N}^{(q)}(\bb{A}^d)$ be a finite subfamily of supports.  Then the canonical map $E^{Z}(\bb{A}^d) \to E^{F^*Z}(\bb{A}^d)$ is homotopic to the composition $E^{Z}(\bb{A}^d) \xrightarrow{\phi^*} E^{\phi^{-1}(Z)}(\bb{A}^d) \to E^{F^*Z}(\bb{A}^d)$.

In particular, the composition
$$E^{Z}(\bb{A}^d) \to E^{F^*Z}(\bb{A}^d) \to E^{F^*Z \setminus \phi^{-1}(Z) }(\bb{A}^d \setminus \phi^{-1}(Z) ) $$
is nullhomotopic.

\end{proposition}

\begin{proof} The map $I_n \times |E^{\Phi_n^{-1}(p^{-1}(Z))}(\Delta^n \times \bb{A}^d \times \A \times \bb{A}^q)| \to |E^{F^*Z}(\Delta^n \times \bb{A}^d \times \bb{A}^q)|$ given by $(j,s) \mapsto f_j^*s$ induces a map
$$\gamma : ||I_\bullet \times (n \mapsto |E^{\Phi_n^{-1}(p^{-1}(Z))}(\Delta^n \times \bb{A}^d \times \A \times \bb{A}^q)| ) || \to ||E^{F^*Z}(\bb{A}^d)||.$$  

We have morphisms of spectra
$$I_\bullet \times E^Z(\bb{A}^d) \xrightarrow{1 \times p^*} I_\bullet \times E^{p^{-1}(Z)}(\bb{A}^d \times \A) \xrightarrow{1 \times \Phi^*} I_\bullet \times E^{\Phi^{-1}(p^{-1}(Z))}(\bb{A}^d \times \A).$$

Let $\beta: I_\bullet \times E^Z(\bb{A}^d) \to E^{F^*Z}(\bb{A}^d)$ denote the composition $\gamma \circ (1 \times \Phi^*) \circ (1 \times p^*)$; we suppress the identification of the Segal realization with the usual geometric realization.  Since $f_{0_n} = i_0$ and $p \circ \Phi_\bullet \circ i_0 = p \circ i_0 \circ \phi_\bullet = \phi_\bullet$, the composition $0_\bullet \times E^Z(\bb{A}^d) \into I_\bullet \times E^Z(\bb{A}^d) \xrightarrow{\beta} E^{F^*Z}(\bb{A}^d)$ is the map $\phi^*$ followed by the inclusion $E^{\phi^{-1}Z}(\bb{A}^d) \to E^{F^*Z}(\bb{A}^d)$.  Since $p \circ \Phi_\bullet \circ i_1 = p \circ i_1 = \id$, the composition $1_\bullet \times E^Z(\bb{A}^d) \into I_\bullet \times E^Z(\bb{A}^d) \xrightarrow{\beta} E^{F^*Z}(\bb{A}^d)$ is the inclusion.  Since the inclusions $0_\bullet \times E^Z(\bb{A}^d), 1_\bullet \times E^Z(\bb{A}^d) \into I_\bullet \times E^Z(\bb{A}^d)$ are homotopic, the result follows. \end{proof}

\begin{corollary} \label{nullhomotopy pairs} Let $Z' \into Z \subset S_{\leq N}^{(q)}(\bb{A}^d)$ be finite subfamilies of supports.  Then the composition
$$E^{Z \setminus Z' }(\bb{A}^d \setminus Z') \to E^{F^*Z \setminus F^*Z' }(\bb{A}^d \setminus F^*Z'  ) \to E^{F^*Z \setminus F^* Z' \cup \phi^{-1}(Z) }(\bb{A}^d \setminus F^* Z' \cup \phi^{-1}(Z) ) $$
is nullhomotopic.

\end{corollary}

\begin{proof} The nullhomotopy established in Proposition $\ref{no new supports}$ holds for $Z'$ and $Z$ separately, hence it holds for the cofiber.  \end{proof}

\begin{corollary} Let $S_1, S_2$ be successive support conditions in the tower $\ref{support tower}$.  Then the canonical morphism $E^{S_1}(\bb{A}^d) \to E^{S_2}(\bb{A}^d)$ is a weak equivalence. \end{corollary}

\begin{proof} The cofiber of the map $E^{S_1}(\bb{A}^d) \to E^{S_2}(\bb{A}^d)$ is the colimit of the cofibers $E^{Z_2 \setminus Z_1 \cap Z_2}(\bb{A}^d \setminus Z_1)$ as $Z_i$ varies over the support condition $S_i(\bb{A}^d)$.  Corollary $\ref{nullhomotopy pairs}$ with $Z' = Z_1 \into Z_1 \cup Z_2 = Z$ and Remark $\ref{move preserves tower}$ show that any term in the colimit admits a nullhomotopic map to a further term in the colimit.  Thus the cofiber is contractible. \end{proof}

\subsection{Extension to smooth affine varieties} \label{general affine}

Since our geometric constructions always involve $X$, those subschemes that are independent of $X$ (i.e., pulled back from $\Delta^n \times \bb{A}^q$) will not be affected by our constructions.  Hence we need a different argument to exclude the possibility that they contribute to a discrepancy between $E^Q(X)$ and $E^{(q)}(X)$.  We say a codimension $q$ subscheme $Z \into \Delta^n \times X \times \bb{A}^q$ is \textit{induced} if there exists a codimension $q$ subset $Y \into \Delta^n \times \bb{A}^q$ such that $Z \subset p^{-1} (Y)$, where $p:= pr_{13} : \Delta^n \times X \times \bb{A}^q \to \Delta^n \times \bb{A}^q$ denotes the projection.  (See \cite[Defn.~5.3.1]{MLChow}.)  In the colimit such an induced subscheme contributes through the subscheme $p^{-1}(Y)$.  

Let $p^{-1}(S)(X)$ denote the support condition generated by subschemes of the form $p^{-1}(Z)$, where $Z \in S(\spec k)$.  The support condition $p^{-1}(S)(X) \subset S(X)$ consists of the induced supports.

\begin{lemma} \label{induced supports} Let $S_1, S_2$ be successive support conditions in the tower $\ref{support tower}$, and suppose $X$ is a smooth $k$-scheme.  Then the canonical map $E^{p^{-1}(S_1)}(X) \to E^{p^{-1}(S_2)}(X)$ is a weak equivalence. \end{lemma}

\begin{proof} This follows from the construction and use of the moving endomorphism $\phi_\bullet$ and the homotopy $\Phi_\bullet$ in Propositions $\ref{FS 8.4}$ and $\ref{no new supports}$: given a finite set $Z$ of codimension $q$ cycles on $\Delta^n \times \bb{A}^q$ intersecting all the faces properly (and possibly satisfying some further condition as specified by $\ref{support tower}$), we can find a pseudo-endomorphism moving every member of $Z$ to a cycle quasi-finite over $\Delta^n$, and a homotopy between the identity and this pseudo-endomorphism.  (This is the $d=0$ case of the previous section.)  Then the construction of the homotopy in the proof of Proposition $\ref{no new supports}$ can be carried out with $\bb{A}^d$ replaced by $X$ and any support $Y$ replaced by $p^{-1}(Y)$.  Finally we note that $Z \mapsto F^*Z$ takes induced supports to induced supports. \end{proof}

\begin{remark} \label{induced decomp} Any subfamily of supports $Z =\{ Z^n_i \} \subset S_n^{(q),e,f}(X)$ (here the $Z^n_i$ are integral $k$-schemes) admits a canonical decomposition $Z = \Ind(Z) \cup Z'$, where $\Ind(Z) \subset p^{-1}S_n^{(q),e,f}(X)$ is the union of the induced components, and $Z'$ consists of the remaining (noninduced) components $Z^n_i$.  Induced supports are preserved by pushforward and pullback by finite flat morphisms.  \end{remark}

\begin{proposition} \label{birationality lemma} Let $X$ be a smooth equidimensional $k$-scheme that is affine and of dimension $d$.  Let $\{Z_i^n \} \subset S_n^{(q), e, f}(X)$ be a finite subfamily of supports, and suppose none are induced.  Then there exists a finite morphism $f \colon X \to \bb{A}^d$ and supports $\{ {Z_i^n}' \} \subset S_n^{(q)}(X)$ satisfying:

\begin{enumerate}
\item the restriction of $1 \times f \times 1$ to $Z^n_i$ is birational onto its image;
\item $Z^n_i \cap {Z^n_i}'$ contains the exceptional locus $\Exc (1 \times f \times 1 |_{Z^n_i})$; and
\item ${Z^n_i}' \in S_n^{(q), e, f+1}(X)$, if $f \leq d+n-1$; or ${Z^n_i}' \in S_n^{(q), e-1, 0}(X)$, if $f = d+n$.
\end{enumerate}

\end{proposition}

\begin{proof} The birationality follows from \cite[Lemma 5.3.3]{MLChow} with $\Delta^p$ replaced by $\Delta^n \times \bb{A}^q$, where it is necessary to assume $k$ is infinite.  For simplicity of notation we assume $Z = Z^n_i$ is a single subvariety and we write $Z' = {Z^n_i}'$.  We write also $f = 1 \times f \times 1$.

As a first approximation, we show the exceptional locus $\Exc (f |_Z)$ satisfies the properties demanded of $Z'$, except that it is not dominant over $\Delta^n \times X$.  Then we show $\Exc (f |_Z)$ extends to a dominant support $Z'$ which is quasi-finite over $\Delta^n \times X$ except where $\Exc (f |_Z)$ fails to be quasi-finite.

We use the superscript $m$ to denote the intersection with a face $\Delta^m \into \Delta^n$.  Let $Z_e \into Z$ denote the locus where the fiber dimension of $Z \to \Delta^n \times X$ is $e$, and similarly for ${(Z^m)}_e \into Z^m$ and ${(Z')}_e \into Z'$.  Note we have ${(Z^m)}_e \subseteq {(Z_e)}^m$ and ${(Z')}_e \subseteq Z_e \cap Z'$.

The exceptional locus is determined by the ramification locus of $f$ on $X$ and the double point locus on $Z$.  In any case, the exceptional locus $Z'$ is an ample divisor on $Z$, and it holds that ${(Z')}^m= \Exc (f |_{Z^m})$.  If $f$ is sufficiently general, then for every face $\Delta^m \into \Delta^n$, the intersection of ${(Z')}^m$ with ${(Z^m)}_e$ is proper.

Assume we are in the first case, so the index $f$ can be increased.  Consider any face $\Delta^m \into \Delta^n$.  Then the dimension of the generic fiber of the composition ${(Z')}^m \cap {(Z^m)}_e \subsetneq {(Z^m)}_e \to pr_{12}({(Z^m)}_e)$ must be less than or equal to $e-1$, therefore $pr_{12} ( {( {(Z')}^m \cap {(Z^m)}_e )}_e )$ cannot be dense in  $pr_{12}({(Z^m)}_e)$.  Since ${({(Z')}^m)}_e = {( {(Z')}^m \cap {(Z^m)}_e )}_e$, it follows that $pr_{12} ({({(Z')}^m)}_e )$ cannot be dense in $pr_{12}({(Z^m)}_e)$, hence the codimension of $pr_{12} ({({(Z')}^m)}_e )$ in $\Delta^m \times X$ is strictly larger than that of $pr_{12}({(Z^m)}_e)$.  Hence we increased the index $f$.

In the second case, the image of $Z_e$ is zero-dimensional in $\Delta^n \times X$.  Then $Z' \cap Z_e$ generically has dimensional less than or equal to $e-1$ over the zero-dimensional scheme $pr_{12}(Z_e)$, hence the locus in $Z'$ that could contain an $e$-dimensional fiber itself is of dimension less than or equal to $e-1$.  Hence $Z'$ cannot contain any $e$-dimensional fibers (i.e., ${(Z')}_e = \emptyset$) and we decreased the index $e$. 

Since the support conditions are compatible with pushforward and pullback by finite flat morphisms ($\ref{elem support properties}(\ref{elem support properties 2},\ref{elem support properties 3})$), it suffices to show $f( \Exc (f |_Z) )$ extends to a support that is dominant over $\Delta^n \times \bb{A}^d$ and quasi-finite except along the locus where $f( \Exc (f |_Z) )$ fails to be quasi-finite.  Now $f( \Exc (f |_Z) )$ is supported over a divisor $D \into \Delta^n \times \bb{A}^d$, and this is true on all faces $\Delta^m \into \Delta^n$.  Now take an arbitrary extension of $f( \Exc (f |_Z) )$ to a codimension $q$ dominant support (there is at least one, namely $f(Z)$), and apply Suslin's moving lemma to this support and require the moving morphism to be the identity on the divisor $D$.  Now for $Z'$ take the pullback via $f$ of the moved extension. \end{proof}
  
\begin{remark} For an explicit construction of a dominant extension that is quasi-finite except over $D$, see Lemma $\ref{extend from divisor}$, where such an extension is constructed in the projective case.  The construction there is an application of Suslin's method of proof and works equally well in the affine case.  \end{remark}

\begin{proposition}  \label{affine tower collapse} Let $X$ be a smooth equidimensional $k$-scheme that is affine and of dimension $d$.  Suppose $E \in {\bf{Spt}}(k)$ satisfies Nisnevich excision.  Let $S_1, S_2$ be successive support conditions in the tower $\ref{support tower}$.  Then the canonical map $E^{S_1}(X) \to E^{S_2}(X)$ is a weak equivalence.
\end{proposition}

\begin{proof} \textit{Step 1.  Eliminate induced supports.} We have a commutative diagram of morphisms of spectra:
$$\xym{E^{p^{-1}(S_1)}(X) \ar[r] \ar[d] & E^{p^{-1}(S_2)}(X) \ar[d] \\
E^{S_1}(X) \ar[r] \ar[d] & E^{S_2}(X) \ar[d] \\
E^{S_1 \setminus p^{-1}(S_1)}(X) \ar[r] & E^{S_2 \setminus p^{-1}(S_2)}(X) \\ }$$
in which the top arrow is a weak equivalence by Lemma $\ref{induced supports}$.  Hence to prove the proposition it suffices to show the bottom arrow is a weak equivalence, i.e., the cofiber of the bottom arrow is contractible.  By Remark $\ref{induced decomp}$ this cofiber is the homotopy colimit of spectra of the form $E^{Z_2 \setminus Z_1 \cup (Z_2 \cap T )}(X \setminus Z_1 \cup T)$, where $Z_1 \subset S_1(X); Z_2 \subset S_2(X); T \subset p^{-1}(S_2)(X)$ are supports satisfying $\Ind(Z_2) = \Ind(Z_1) = \emptyset$.

\textit{Step 2.  Apply excision.} Suppose given finite sets of supports $Z_1 \subset S_1(X); Z_2 \subset S_2(X); T \subset p^{-1}(S_2)(X)$, and for simplicity of notation assume $Z_1 \into Z_2$.  By Step 1 we may suppose neither of the $Z_i$s is induced.  We will find $Y_1 \subset S_1 \setminus p^{-1}(S_1)(X); Y_2 \subset S_2 \setminus p^{-1}(S_2)(X); T' \subset p^{-1}(S_2)(X)$ such that the induced map
$$E^{Z_2 \setminus Z_1 \cup (Z_2 \cap T)}(X \setminus Z_1 \cup T ) \to E^{Z_2 \cup Y_2 \setminus Z_1 \cup Y_1 \cup ((Z_2 \cup Y_2) \cap T') }(X \setminus Z_1 \cup Y_1 \cup T')$$ is nullhomotopic.

Now choose a finite morphism $f : X \to \bb{A}^d$ as in Proposition $\ref{birationality lemma}$ with $Z_2 = \{ Z^n_i \}$, and let $Z_2'$ denote the supports lying in the smaller support condition $S_1(X)$ and containing the exceptional locus of $1 \times f \times 1 |_{Z_2}$.  

By \cite[Lemma 6.1.1]{MLChow} we have a weak equivalence
$$E^{Z_2 \setminus Z_2 \cap Z_2'} (X \setminus Z_2') \cong E^{f(Z_2) \setminus f(Z_2 \cap Z_2')}(\bb{A}^d \setminus f(Z_2')),$$
and similarly after the removal of $Z_2 \cap T$.  For the same reason we have a weak equivalence
$$E^{Z_1 \setminus Z_1 \cap Z_2'}(X \setminus Z_2') \cong E^{f(Z_1) \setminus f(Z_1 \cap Z_2')}(\bb{A}^d \setminus f(Z_2')).$$

By comparing the cofibers we deduce the map
$$E^{Z_2 \setminus (Z_2 \cap (T \cup Z_2')) \cup Z_1}(X \setminus T \cup Z_2' \cup Z_1) \xrightarrow{\sim} E^{f(Z_2) \setminus f( (Z_2 \cap (T \cup Z_2')) \cup Z_1)}(\bb{A}^d \setminus f( T \cup Z_2' \cup Z_1 ) )$$ is also a weak equivalence.

\textit{Step 3.  Obtain a nullhomotopy.} We use the notations $Z := Z_2, W := f(Z), Z' :=  (Z \cap T) \cup (Z \cap Z_2') \cup Z_1$, and $W' := f(Z')$; we also write $q = \phi \circ f$.  Note we have a decomposition into disjoint simplicial subschemes $$f^{-1}(f(Z \setminus Z \cap Z_2' )) = (Z \setminus Z \cap Z_2') \coprod (Z^+ \setminus {e(Z}^{+}))$$ and therefore a decomposition $$f^{-1}(f(Z \setminus Z')) = (Z \setminus Z') \coprod (Z^+ \setminus Z^{'+} ).$$

Thus we obtain a commutative diagram, where the top row involves evaluation on open subschemes of $\Delta^\bullet \times X \times \bb{A}^q$ and the bottom row on open subschemes of $\Delta^\bullet \times \bb{A}^d \times \bb{A}^q$, in which the bottom arrow is nullhomotopic.
\begin{equation} \label{excision diagram} \xym{E^{Z \setminus Z'} \oplus E^{Z^+ \setminus Z^{'+}} \ar[r] & E^{f^{-1}(F^*W) \setminus q^{-1}(W) \cup f^{-1}(F^*W') } \\
E^{W \setminus W'}  \ar[u]^-{f^*} \ar[r] & E^{  F^*W \setminus \phi^{-1}W \cup F^*W' } \ar[u]^-{f^*} \\ }
\end{equation}

After discarding the incidence $f^{-1}(F^*W) \cap Z^+$ in the top row, we have a commutative diagram:
\begin{equation} \label{excision diagram 2} \xym{ E^{Z \setminus Z'}  \ar[r] & E^{Z \cup f^{-1} (R^*W) \setminus q^{-1}(W) \cup f^{-1}(F^*W') \cup (f^{-1}(R^*W) \cap Z^+)} \\ 
E^{Z \setminus Z'} \oplus E^{Z^+ \setminus Z^{'+}} \ar[u]^{p_1} \ar[r] & E^{Z \cup f^{-1} (R^*W) \setminus q^{-1}(W) \cup f^{-1}(F^*W') \cup (f^{-1}(R^*W) \cap Z^+)} \oplus E^{Z^+ \setminus Z_s^+ }\ar[u]_-{p_1}  \\ } \end{equation}

The space for the top right entry is $X \setminus q^{-1}(W) \cup f^{-1}(F^*W') \cup Z^+$, and $Z_s^+$ denotes the intersection $Z^+ \cap (q^{-1}(W) \cup f^{-1}(F^*W') \cup (f^{-1}(R^*W))$.  In Step 2 we observed the induced map $E^{W \setminus W'} \to E^{Z \setminus Z'}$ is a weak equivalence.  This implies the map in the top row of $\ref{excision diagram}$ factors (up to homotopy) through the nullhomotopy $E^{W \setminus W'}  \to E^{  F^*W \setminus \phi^{-1}W \cup F^*W' } $, therefore the map in the top row of $\ref{excision diagram 2}$ is also a nullhomotopy.

\textit{Step 4.  Shrink supports.}  While we have $q^{-1}(W) \cup f^{-1}(F^*W') \subset S_1(X)$, it does not seem possible to guarantee that $Z^+$ belongs to $S_1(X)$.  However, we can combine excision with the fact that $Z$ and $Z^+$ have intersection contained in $Z_2'$, which does lie in $S_1(X)$.  More precisely, the fact that
$$E^{Z \setminus Z'}  \to E^{Z \cup f^{-1} (R^*W) \setminus q^{-1}(W) \cup f^{-1}(F^*W') \cup (f^{-1}(R^*W) \cap Z^+)}$$
is nullhomotopic implies $E^{Z \setminus Z'}$ factors through the homotopy fiber $fib$ of the canonical map
$$E^{Z \cup f^{-1} (R^*W) \setminus q^{-1}(W) \cup f^{-1}(F^*W')} (X \setminus q^{-1}(W) \cup f^{-1}(F^*W') )\to$$
$$\to E^{Z \cup f^{-1}(R^*W) \setminus q^{-1}(W) \cup f^{-1}(F^*W') \cup (f^{-1}(R^*W) \cap Z^+)}(X \setminus q^{-1}(W) \cup f^{-1}(F^*W') \cup Z^+).$$
(In the first space, we did not discard $Z^+$.)  Now $fib$ is nothing but
$$E^{f^{-1}(R^*W) \cap Z^+ \setminus ((f^{-1}(R^*W) \cap Z^+) \cap (q^{-1}(W) \cup f^{-1}(F^*W'))  } (X \setminus q^{-1}(W) \cup f^{-1}(F^*W') ) ,$$
which may be more transparent if one ignores the removal of $q^{-1}(W) \cup f^{-1}(F^*W')$.  Since the map $E^{Z \setminus Z'} \to fib$ is induced by canonical ``inclusion of support" maps, this map too is compatible with adding supports and shrinking $X$.  In particular, we can remove $Z$.  But ${f^{-1}(R^*W) \cap Z^+ \setminus ((f^{-1}(R^*W) \cap Z^+) \cap (q^{-1}(W) \cup f^{-1}(F^*W') ) }$ is disjoint from (the restriction of) $Z$ on $X \setminus q^{-1}(W) \cup f^{-1}(F^*W')$, hence by excision the homotopy fiber $fib$ is weakly equivalent to the space where $Z$ has also been removed.  Thus the map
$$E^{Z \setminus Z'} \to E^{Z \cup f^{-1} (R^*W) \setminus q^{-1}(W) \cup f^{-1}(F^*W')} (X \setminus q^{-1}(W) \cup f^{-1}(F^*W'))$$
is nullhomotopic, and therefore so is its precomposition with $E^{Z_2 \setminus Z_1} \to E^{Z \setminus Z'}$.  Since this map occurs in the colimit of the cofiber of the map $E^{S_1 \setminus p^{-1}(S_1)}(X) \to E^{S_2 \setminus p^{-1}(S_2)}(X)$, the proof is complete.  \end{proof}

In the proof we used the slightly abusive notation $E^{Z_1 \coprod Z_2}(X) = E^{Z_1}(X) \oplus E^{Z_2}(X)$.  Since there is a canonical fibration sequence $E^{Z_1}(X) \to E^{Z_1 \coprod Z_2}(X) \to E^{Z_2}(X \setminus Z_1)$ and $E^{Z_2}(X) \cong E^{Z_2}(X \setminus Z_1)$ (and similarly with the roles of $Z_1, Z_2$ reversed), this notation is accurate as far as the homotopy groups are concerned.

\begin{corollary} \label{main comparison affine} Let $X$ be a smooth equidimensional affine $k$-scheme.  Suppose $E \in {\bf{Spt}}(k)$ satisfies Nisnevich excision.  Then $\alpha_X: E^{Q}(X) \to E^{(q)}(X)$ is a weak equivalence. \end{corollary}

\begin{remark} \label{ML funct} This result should be compared with the version of Chow's moving lemma used by Levine to construct functorial models of $E^{(q)}$ in case $B = \spec k$.  In this remark we temporarily switch notation to consider codimension $\geq q$ support conditions on cosimplicial schemes of the form $\Delta^\bullet \times X$, not just those of the form $\Delta^\bullet \times X \times \bb{A}^q$.  For a finite set $\cal{C}$ of locally closed subsets $C$ of $X$ and a function $e : \cal{C} \to \bb{N}$, one can define a support condition $S^{(q)}_{\cal{C},e}(X)$ consisting of subschemes $Z \into \Delta^\bullet \times X$ intersecting all the faces in codimension $q$, and with the improperness of the incidences of the form $Z \cap (\Delta^\bullet \times C)$ controlled by the function $e$ \cite[Defn.~2.6.1]{MLChow}.  We have $S^{(q)}_{\cal{C},e}(X) \subset S^{(q)}(X)$ and we may form the spectrum with supports $E^{(q)}(X)_{\cal{C},e}$.  

Levine's result is that, for $j : U \subset X$ an affine open subscheme of $X \in {\bf{Sm}} / k$, the canonical map $E^{(q)}(U)_{j^*\cal{C}, j^* e} \to E^{(q)}(U)$ is a weak equivalence \cite[Thm.~2.6.2(2)]{MLChow}.  Combined with the homotopy invariance $E^{(q)}(X) \xrightarrow{\sim} E^{(q)}(X \times \A)$ for $X \in Sch / k$ \cite[Thm.~3.3.5]{MLChow}, we see that, for a smooth affine $k$-scheme $X$, the canonical inclusion of supports map is a weak equivalence $\iota_{\cal{C}, e} : E^{(q)}(X \times \bb{A}^q)_{p^*\cal{C}, p^*e} \xrightarrow{\sim} E^{(q)}(X \times \bb{A}^q)$.  (Here $p : X \times \bb{A}^q \to X$ is the projection.)  Since a quasi-finite support intersects properly any set of the form $\Delta^n \times C \times \bb{A}^q$, for any $\cal{C}, e$ we have an inclusion of support conditions $Q(X) \subset S^{(q)}_{p^*\cal{C}, p^*e}(X \times \bb{A}^q)$.  Thus our result says the source of the weak equivalence $\iota_{\cal{C},e}$ may be replaced by the spectrum with supports in the smaller (and functorial) condition $Q(X)$. \end{remark}

\begin{remark} \label{Nis exc remark} The presheaf $E^Q$ admits a canonical morphism to its Nisnevich-fibrant replacement $E^Q_{Nis}$.  The functorial model $\tilde{E}^{(q)}$ of $E^{(q)}$ constructed in \cite[Thms.~7.4.1,7.5.4]{MLChow} satisfies Zariski excision by \cite[Thm.~3.2.1]{MLHtyCon}, hence Nisnevich excision if the base field is perfect.  Hence the canonical morphism $E^Q \to \tilde{E}^{(q)}$ factors through $E^Q_{Nis}$.  We take sectionwise fibrant replacements of $E^Q_{Nis}$ and $\tilde{E}^{(q)}$ \cite[Lemma 2.4(a)]{Jardine} and suppress it from the notation.

The morphism $E^Q \to \tilde{E}^{(q)}$ induces a weak equivalence on affine $k$-schemes by Corollary $\ref{main comparison affine}$, hence induces a weak equivalence on Nisnevich stalks (i.e., is a Nisnevich-local weak equivalence).  The morphism $E^Q \to E^Q_{Nis}$ is also a Nisnevich-local weak equivalence.  Therefore $E^Q_{Nis} \to \tilde{E}^{(q)}$ is a Nisnevich-local weak equivalence.  Furthermore, both $E^Q_{Nis}$ and $\tilde{E}^{(q)}$ are sectionwise fibrant presheaves of spectra satisfying Nisnevich excision.  A result of Thomason \cite[Prop.~2.3]{Thomason} implies $E^Q_{Nis} \to \tilde{E}^{(q)}$ is a sectionwise weak equivalence, in particular is a sectionwise weak equivalence on affine $k$-schemes.  But then also $E^Q \to E^Q_{Nis}$ must be a sectionwise weak equivalence on affine $k$-schemes.  In this sense $E^Q$ has the correct value on affine $k$-schemes.

Corollary $\ref{main comparison affine}$ implies $E^Q$ and $E^{(q)}$ have the same stalks, hence Thomason's spectrum level Godement construction \cite[Sect.~1]{Thomason} on $E^Q$ is sectionwise weakly equivalent to $E^{(q)}$.  Corollary $\ref{main comparison affine}$ also shows an affine open cover is sufficiently fine that the associated \v{C}ech hypercohomology spectrum of $E^Q$ is weakly equivalent to the Nisnevich-fibrant model.  \end{remark}

\begin{remark} \label{FS comparison} For algebraic $K$-theory, Friedlander-Suslin showed the map $\alpha_X$ is a weak equivalence for $X$ the spectrum of a field \cite[Thm.~8.6]{FS}.  Identifying the cofiber of $K^{(q+1)} \to K^{(q)}$ with motivic cohomology, one obtains a strongly convergent spectral sequence from the resulting tower.  The cofiber of the morphism $K^{Q,q+1}(X) \to K^{Q, q}(X)$ was also identified with motivic cohomology, for semilocal $X$ \cite[Thm.~11.5]{FS}; and the resulting tower was shown to yield a strongly convergent spectral sequence \cite[Thm.~13.13]{FS}.  Therefore both towers determine strongly convergent spectral sequences with the same initial page on semilocal schemes, hence $\alpha_X : K^{Q,q}(X) \to K^{(q)}(X)$ is a weak equivalence for semilocal $X$.  If $X$ is the semilocalization of a quasi-projective variety at a finite set of points, the semilocal stalk can be computed over affine open neighborhoods.  Since algebraic $K$-theory satisfies Nisnevich excision, Corollary $\ref{main comparison affine}$ recovers the comparison results of Friedlander-Suslin. \end{remark}

\section{Moving supports on projective varieties}

We will analyze the proof of Suslin's moving lemma \cite[Thm.~1.1]{Sus} in order to understand the closures of cycles in $\Delta^n \times \bb{A}^d \times \bb{A}^q$ on $\Delta^n \times \bb{P}^d \times  \bb{A}^q$.  The main point is to achieve a move on the finite part (i.e., $\bb{A}^d \subset \bb{P}^d$) that does not disturb the situation at infinity (i.e., $\bb{P}^{d-1} \into \bb{P}^d$).  For a given collection of cycles lying in $S_n^{(q), e,f}(\bb{P}^d)$, a sufficiently general hyperplane will intersect the family in a lower filtration level of $\ref{support tower}$.  We will show the incidence with the hyperplane at infinity can be controlled while moving the finite part to a quasi-finite cycle.  Therefore the collection can be moved to a lower filtration level and the essential geometric input to the affine case extends to the projective case.

\begin{example}  We start with an example illustrating the method of Suslin's moving lemma.  

\subtheorem \textbf{(0-simplices)}  Suppose given a divisor $D^0 = V(f(y,x)) \into \A_y \times \A_x$ with $f(y,x) = y g(x)$.  Then $D^0$ is quasi-finite over $(\A \setminus 0)_y$, but it contains the fiber $0_y \times \A_x$.  We want to move $D^0$ to a divisor that is quasi-finite over $\A_y$, i.e., we seek an $\A_x$-morphism $\phi: \A_y \times \A_x \to \A_y \times \A_x$ such that $\phi^{-1}(D^0) \to \A_y$ is quasi-finite.  We set $\phi(y,x) = (y + p(x), x)$ where $p$ is a nonconstant polynomial.  Then $\phi^{-1}(D^0)$ is defined by the vanishing of $p(x)g(x) + yg(x)$.  This expression does not vanish identically for any value of $y$, hence all of the fibers ${(\phi^{-1}(D^0) )}_{y_0}$ are finite.

We notice two features of this example.  First, the construction extends to $\oline{D^0} = Z(Y_1 g(x) ) \into \bb{P}^1_{Y_0, Y_1} \times \A_x$: we have the morphism $\oline{\phi} : \bb{P}^1_{Y} \times \A_x \to \bb{P}^1_{Y} \times \A_x$ given by $[Y_0 : Y_1], x \mapsto [Y_0 : Y_1 + p(x) Y_0], x$, and ${\oline{\phi}}^{-1}(\oline{D^0}) = Z((Y_1 + p(x) Y_0) g(x))$ is quasi-finite over $\bb{P}^1_{Y}$.  Second, the maps $\phi$ and $\oline{\phi}$ are related to the identity map by homotopies $\Phi$ and $\oline{\Phi}$:
$$\oline{\Phi} : \A_t \times \bb{P}^1_Y \times  \A_x \to \A_t \times \bb{P}^1_Y \times  \A_x$$
$$(t, [Y_0 : Y_1], x) \mapsto (t, [Y_0: Y_1 + (1-t)p(x) Y_0], x)$$
Let $p :  \A_t \times \bb{P}^1_Y \times  \A_x \to \bb{P}^1_Y \times \A_x$ denote the projection.  The divisor $\scr{D}^0 := {\oline{\Phi}}^{-1}( p^{-1} (\oline{D^0}))$, cut out by the equation $(Y_1 + p(1-t)Y_0) g = 0$, has the following properties:
\begin{enumerate}
\item $\scr{D}^0 |_{t=0}= {\oline{\phi}}^{-1}(\oline{D^0})$,
\item $\scr{D}^0 |_{t=1} = \oline{D^0}$, and
\item $\scr{D}^0$ is quasi-finite over $(\A \setminus 1)_t \times \bb{P}^1_Y$, in fact over $ \A_t \times \bb{P}^1_Y \setminus ( 1, [1:0] )$.
\end{enumerate}
To obtain $\Phi$, we can restrict to $\A_{Y_0 \neq 0} \subset \bb{P}^1_Y$.

\subtheorem \textbf{(1-simplices)} We use the coordinates given by the identification $\Delta^1 = \spec (k[t_0,t_1] /(t_0 + t_1=1))$.  Now let $D^1 \into \Delta^1 \times \A_y \times \A_x$ be the divisor defined by the vanishing of $t_1(t_0 + yg) + t_0(t_1 +yh)$; here $g,h \in k[x]$ are polynomials that (for simplicity) are not scalar multiples of each other.  Then $D^1 |_{t_0=0} = V(yg)$ and $D^1 |_{t_1=0} = V(yh)$, and $D^1$ is quasi-finite over $\Delta^1 \times \A_y \setminus \{((0,1),0) \cup ((1,0),0) \}.$

By our analysis of the 0-simplices, the map $\phi: \A_y \times \A_x \to \A_y \times \A_x$ given by $\phi(y,x) = (y+p(x),x)$  moves both $V(yg)$ and $V(yh)$ to quasi-finite cycles, hence $\phi_0 = 1 \times \phi : V(t_0t_1) \times \A_y \times \A_x \to V(t_0t_1) \times \A_y \times \A_x$ moves $D^1 |_{t_0=0}$ and $D^1 |_{t_1=0}$ to quasi-finite cycles.  We wish to extend $\phi_0$ to the algebraic $1$-simplex, i.e., we seek an $\A_x$-morphism $\phi_1 : \Delta^1 \times \A_y \times \A_x \to \Delta^1 \times \A_y \times \A_x$ such that:
\begin{itemize}
\item $\phi_1 |_{V(t_0t_1)} = \phi_0$, and
\item $\phi_1^{-1}(D^1)$ is quasi-finite over $(\Delta^1 \setminus V(t_0 t_1) ) \times \A_y$.
\end{itemize}
(From these it follows that $\phi_1^{-1}(D^1)$ is quasi-finite over $\Delta^1 \times \A_y$.)  Choose homogeneous forms (in $x$) $a,b$ satisfying $\deg(a) = \deg(b) > \deg (p)$, i.e., $a$ and $b$ are scalar multiples of $x^n$ with $n > \deg(p)$.  Then we define:
$$\phi_1(t_0,t_1,y,x) = (t_0 + t_0 t_1 a(x), t_1 - t_0 t_1a(x), y + p(x) + t_0 t_1 b(x), x).$$
One checks immediately the first property, that $\phi_1$ restricts to $\phi_0$ on the vertices of $\Delta^1$.  To check the second property, observe that $\phi_1^{-1}(D^1)$ is defined by the vanishing of
$$ f_{\phi_1^{-1}(D^1)} := (t_1 - t_0 t_1a)(t_0 + t_0t_1a + g(y + p + t_0t_1b)) + (t_0 + t_0t_1a)(t_1 - t_0t_1a + h(y + p + t_0 t_1 b)),$$
and we just need to show that, away from $t_0t_1 =0$, there are no values of $(t_0,t_1,y)$ making $f_{\phi_1^{-1}(D^1)}$  identically zero.  To see this, we notice the term of largest $x$-degree is ${(t_0 t_1)}^2 abc$.  (Here and in a few paragraphs, $c$ denotes $\lf_x (h-g)$, the leading $x$-term of $h-g$.)  Since this term is nonzero when $t_0t_1 \neq 0$, it follows that $f_{\phi_1^{-1}(D^1)}$ is not the zero polynomial on this locus.

Similar to the $0$-simplices, the situation closes up nicely to an endomorphism $\oline{\phi_1}$ on $\Delta^1 \times \bb{P}^1_Y \times \A_x$: we have
$$\oline{\phi_1}(t_0,t_1,[Y_0:Y_1],x) = (t_0 + t_0 t_1 a(x), t_1 - t_0 t_1a(x), [Y_0: Y_1 + (p(x) + t_0 t_1 b(x))Y_0], x)$$
moving $\oline{D^1} :=Z(t_1(Y_0 t_0 + Y_1g) + t_0(Y_0t_1 + Y_1h))$ to the $\Delta^1 \times \bb{P}^1_Y$-quasi-finite divisor ${\oline{\phi_1}}^{-1}(\oline{D^1})$ (the quasi-finiteness along $Y_0 = 0$ follows from the assumption that $g$ is not a scalar multiple of $h$).
Furthermore we have a homotopy $\oline{\Phi_1}$ between $\oline{\phi_1}$ and the identity, the endomorphism $\oline{\Phi_1}$ of $\Delta^1 \times \A_t \times \bb{P}^1_Y \times  \A_x$ sending $(t_0, t_1, t, [Y_0 : Y_1], x)$ to
$$(t_0 + (1-t)t_0 t_1 a(x), t_1 - (1-t) t_0 t_1a(x), t, [Y_0: Y_1 + (1-t)(p(x) + t_0 t_1 b(x))Y_0], x).$$

Therefore the divisor $ \scr{D}^1 := {\oline{\Phi_1}}^{-1}(p^{-1}(\oline{D^1}))$ (again $p$ is the projection away from $\A_t$) enjoys the following properties:
\begin{enumerate}
\item $\scr{D}^1 |_{t=0}= {\oline{\phi_1}}^{-1}(\oline{D^1})$,
\item $\scr{D}^1 |_{t=1} = \oline{D^1}$, and
\item $\scr{D}^1$ is quasi-finite over $\Delta^1 \times (\A \setminus 1)_t \times \bb{P}^1_Y$.
\end{enumerate}

To see the last property, we argue as follows.  Along $Y_0 \neq 0$, we calculate the term of largest $x$-degree in the function whose vanishing defines ${\oline{\Phi_1}}^{-1}(p^{-1} (\oline{D^1}))$; we find ${((1-t)t_0t_1)}^2 abc$ and conclude the quasi-finiteness on $(1-t)t_0t_1 \neq 0$.  Along $Y_0 = 0$, the leading $x$-term is $Y_1(1-t)t_0t_1ac$, so the conclusion holds over $(\Delta^1 \setminus V(t_0 t_1) ) \times (\A \setminus 1)_t \times \bb{P}^1_Y$.  Along the locus $t_0 t_1 = 0$ we may apply our analysis of the $0$-simplices.  The conclusion follows.

Finally we analyze the pullback of $\scr{D}^1$ along the closed immersion $f_{01} : \Delta^1 \into \Delta^1 \times \A_t$ determined by the equation $t=t_0$.  The immersion $f_{01}$ is the unique linear morphism sending $(0,1)$ to $((0,1),0)$, and $(1,0)$ to $((1,0),1)$.  By the third property of $\scr{D}^1$, the non quasi-finite locus of $f_{01}^* (\scr{D}^1)$ lies over $t=1$, so over $t_0=1$.  The equation defining $f_{01}^* (\scr{D}^1) |_{t_0=1}$ is simply $hY_1 =0$.  For us the relevant point is that the non quasi-finite locus of  $f_{01}^* (\scr{D}^1)$, namely $((1,0),0)$, is contained in the non quasi-finite locus of $\oline{D^1}$.
\end{example}

\subsection{Moving lemma with infinity}

\begin{notation} We use the standard coordinates $t_0, \ldots, t_n$ on $\Delta^n$.  We use coordinates $Y_0, \ldots, Y_d$ on $\bb{P}^d$.  We have the open dense subscheme $\bb{A}^d_0 = \bb{A}^d_{Y_0 \neq 0} \subset \bb{P}^d$ and the complementary hyperplane at infinity $\bb{P}^{d-1}_\infty$, defined by the vanishing of $Y_0$.  Actually we think of $\bb{P}^{d-1}_\infty$ as a hyperplane which is very general with respect to a finite collection of data rather than as a fixed hyperplane in a fixed projective space; in other words we choose coordinates so that $Z(Y_0)$ is very general.  We use coordinates $x_1, \ldots , x_q$ on $\bb{A}^q$. \end{notation}

Given a support $Z \in \Delta^\bullet \times X \times \bb{A}^q$, we would like consider supports of the form $Z \cap (\Delta^\bullet \times H \times \bb{A}^q)$ for $H \into X$ an ample divisor.  However, these fail to be dominant over $\Delta^\bullet \times X$.  In some sense the next two lemmas allow us to consider supports lying over a divisor in $\Delta^\bullet \times X$.  The proof of Lemma $\ref{extend from divisor}$ is based on the proof of Suslin's moving lemma.

\begin{lemma} \label{extend from divisor} Let $W \into \Delta^n \times \bb{P}^d \times \bb{A}^q$ be a support in $S_n^{(q)}(\bb{P}^d)$, and let $D \into \Delta^n \times \bb{P}^d$ be an effective divisor that is not pulled back from $\Delta^n$.  Then there exists a support $W' \into \Delta^n \times \bb{P}^d \times \bb{A}^q$ in $S_n^{(q)}(\bb{P}^d)$ such that $W \cap D = W' \cap D$ set-theoretically, and such that $W' $ is quasi-finite over $\Delta^n \times \bb{P}^d \setminus D$. \end{lemma}
\begin{proof} Let $\{ f_1, \ldots , f_r \}$ be generators for the ideal of $W$ in $\Delta^n \times \bb{P}^d \times \bb{A}^q$, and let $h$ be a generator for the ideal of $D \into \Delta^n \times \bb{P}^d$.  Note that $r \geq q$.  We let $W'$ be defined by the ideal generated by the elements $f'_i := {f_i(t,Y,x)}^{|h|} + {h(t,Y)}^{|f_i|}p_i(x)$ for $i = 1, \ldots, r$.  Here $| \cdot |$ denotes the $Y$-degree of a form, and the $p_i$s are sufficiently general homogeneous forms of the same sufficiently large (compared to the $x$-degree of the $f_i$s) degree.

Set theoretically we have $W' \cap D = W \cap D$.  If $D \not \supset W'$ (i.e., $h(t,Y) \not \in I_{W'}$), then the dimension of $W'$ is $n+d$.  More generally, to show $W'$ dominates $\Delta^n \times \bb{P}^d$, we need to show no element of the form $g(t,Y)$ belongs to the ideal defining $W'$.  Suppose we have an equation $\sum_i a_i(t,Y,x)f'_i(t,Y,x) = g(t,Y)$.  This implies the terms of positive $x$-degree on the left hand side all vanish.  Looking at the largest $x$-degree term, we find the relation $\sum_i a'_i p_i = 0$, where $a'_i$ is the largest $x$-degree term of $a_i$.  Specializing, we obtain a nontrivial $k$-relation among the $p_i$s.  This is impossible if the $p_i$s are chosen generically.

Now we calculate the dimensions of the fibers of the morphism $W' \to \Delta^n \times \bb{P}^d$.  We choose coordinates on $\bb{P}^q$ so that $\bb{A}^q \subset \bb{P}^q$ is the complement of the zero locus of $X_0$.  The closure $\oline{W'}$ of $W'$ in $\Delta^n \times \bb{P}^d \times \bb{P}^q$ is defined by the ideal generated by $X^{n_i}_0 {f_i}^{|h|} + h^{|f_i|} p_i$ for suitable $n_i > 0$ (because the $p_i$s have large $x$-degree), hence the ideal of the infinite part $\oline{W'} \cap (\Delta^n \times \bb{P}^d \times \bb{P}^{q-1}_\infty)$ is generated by the elements $h^{|f_i|}p_i$.  Over $(\delta_0, y_0) \in \Delta^n \times \bb{P}^d \setminus D$, the infinite part of the fiber (i.e., $(\delta_0, y_0) \cap \oline{W'} \cap (\Delta^n \times \bb{P}^d \times \bb{P}^{q-1}_\infty)$) is defined by the ideal generated by $p_1, \ldots, p_r$.  Since $r \geq q$, the infinite part of the fiber is empty if the $p_i$s are chosen so that $Z(p_1, \ldots, p_r) = \emptyset$.  Hence, away from $D$, the fiber of $W'$ itself must be finite.   \end{proof}

\begin{corollary} \label{hyperplane lowers filtration} Let $W \into \Delta^n \times \bb{P}^d \times \bb{A}^q$ be a support in $S^{(q),e,f}_n(\bb{P}^d)$.  Then there exists a hyperplane $H \into \bb{P}^d$ and a support $W' \into \Delta^n \times \bb{P}^d \times \bb{A}^q$ satisfying:
\begin{itemize}
\item $W \cap (\Delta^n \times H \times \bb{A}^q) = W' \cap (\Delta^n \times H \times \bb{A}^q)$; and
\item $W'$ belongs to $S_n^{(q), e,f+1}(\bb{P}^d)$, if $f \leq d+n-1$; or to $S_n^{(q),e-1,0}(\bb{P}^d)$, if $f=d+n$.
\end{itemize}  \end{corollary}

\begin{proof} As a first approximation, we show there exists a hyperplane $H \into \bb{P}^d$ such that $W \cap H := W \cap (\Delta^n \times H \times \bb{A}^q)$ belongs to the smaller support condition, except that it is not dominant over $\Delta^n \times \bb{P}^d$.   Then Lemma $\ref{extend from divisor}$ provides an extension of $W \cap H$ to a dominant support that is quasi-finite except over $H$, and hence $W'$ belongs to the smaller support condition.

We use the superscript $m$ to denote the intersection with a face $\Delta^m \into \Delta^n$.  We let $W_e \into W$ denote the locus where the fiber dimension of $W \to \Delta^n \times \bb{P}^d$ is $e$, and similarly we have ${(W \cap H)}_e \into W \cap H$.  Note we have ${(W \cap H)}_e \subseteq W_e \cap H$.

Suppose we are in the first case.  We choose the hyperplane $H$ so that for every face $\Delta^m \into \Delta^n$, the intersection $pr_{12}({(W^m)}_e) \cap (\Delta^m \times H)$ is proper: since the linear system $|H|$ is basepoint free, $pr_{12}({(W^m)}_e)$ cannot be contained in every $\Delta^m \times H$, hence it is not contained in $\Delta^m \times H$ for $H$ general.  Then the codimension of $pr_{12}({(W^m)}_e) \cap (\Delta^m \times H)$ in $\Delta^m \times \bb{P}^d$ must be strictly larger than that of $pr_{12}({(W^m)}_e)$.  Since clearly $pr_{12} ({({(W \cap H)}^m)}_e) \subseteq pr_{12}({(W^m)}_e) \cap (\Delta^m \times H)$, we increased the index $f$.

In the second case, the image $pr_{12}(W_e)$ is a zero-dimensional subscheme in $\Delta^n \times \bb{P}^d$, hence $pr_2(W_e) \into \bb{P}^d$ is also zero-dimensional.  We can find a hyperplane $H \into \bb{P}^d$ such that $pr_2(W_e) \cap H = \emptyset$.  Since $(W \cap H)_e \subset W_e \cap (\Delta^n \times H \times \bb{A}^d)$, it follows that $(W \cap H)_e = \emptyset$, hence $W' \in S_n^{e-1,0}(\bb{P}^d)$ as desired. \end{proof}

\begin{theorem} \label{projective suslin} Let $W \into \Delta^n \times  \bb{P}^d \times \bb{A}^q$ be a support in $S_n^{(q),e,f}(\bb{P}^d)$.  Then for $i=0, \ldots, n$ there exist $\bb{A}^q$-morphisms $\phi_i : \Delta^i \times \bb{P}^d \times  \bb{A}^q \to \Delta^i \times \bb{P}^d \times  \bb{A}^q$ with the following properties.

\begin{enumerate}
\item For every $0 \leq i \leq n$ and every face $\delta : \Delta^i \into \Delta^n$, we have $\phi_n \circ \delta = \delta \circ \phi_i$.  In other words, the $\{ \phi_i \}$ define a pseudo-endomorphism $\phi_\bullet$ of the truncated cosimplicial scheme $(\Delta^\bullet \times \bb{P}^d \times  \bb{A}^q )_{\leq n}$.
\item The support $\phi_n^{-1}(W) \into \Delta^n \times \bb{P}^d \times  \bb{A}^q$ is quasi-finite over $ \Delta^n \times \bb{A}^d_0$.
\item The support $\phi_n^{-1}(W)$ belongs to $S_n^{(q), e,f+1}(\bb{P}^d)$, if $f \leq d+n-1$; or to $S_n^{(q),e-1,0}(\bb{P}^d)$, if $f=d+n$.
\end{enumerate}

\end{theorem}

We will use the $t=0$ case of the following.

\begin{lemma} \label{move on vertices} Let $V \into \bb{P}^d \times \bb{A}^q$ be a closed subscheme of dimension $d+t$.  Let $\phi : \bb{P}^d \times \bb{A}^q \to \bb{P}^d \times \bb{A}^q$ be defined by $\phi ([Y_0 : \cdots : Y_d], x ) = ([Y_0 : Y_1 + p_1(x) Y_0 : \cdots : Y_d + p_d(x) Y_0], x)$.  Then for sufficiently generic homogeneous forms $p_1, \ldots, p_d$ of the same sufficiently large degree, the fibers of the morphism $\phi^{-1}(V) \to \bb{P}^d$ have dimension $\leq t$ over $\bb{A}^d_{0}$.  \end{lemma}

\begin{proof}

This is essentially a direct consequence of Suslin's argument \cite{Sus} with an empty divisor; we have chosen a particular moving morphism $\phi$ that preserves the situation at infinity.  For completeness we copy the argument of \cite{Sus} adapted to our situation.

For any $G(Y,x) \in k[\bb{A}^q][Y_0, \ldots, Y_d]$, i.e., a homogeneous form in the $Y$s with coefficients in $k[x_1, \ldots , x_q]$, we let $g(y,x)$ denote the dehomogenization with respect to $Y_0$.  Let $V_{\infty, \infty} \into \bb{P}^{d-1} \times \bb{P}^{q-1}$ be the variety obtained by taking the closure of $V_\infty$ in $\bb{P}^{d-1}_\infty \times \bb{P}^q$ , then intersecting with $\bb{P}^{d-1}_\infty \times \bb{P}^{q-1}_\infty$.  Now choose generators $G_j(Y,x)$ for the ideal of $V$ with the property that the forms $\tilde{g_j}:= \lf_x \lf_y g_j(y,x)$ generate the ideal of leading forms, i.e., cut out the ideal of $V_{\infty, \infty}$.

If we choose $p_i$'s homogeneous of the same degree, with $\deg(p_i) > \max (\deg_x g_j)$, then we have $\lf_x( g_j(y_1 + p_1(x), \ldots, y_d + p_d(x) ; x)) = \tilde{g}_j(p_1(x), \ldots, p_d(x);x)$.  Therefore, over $Y_0 \neq 0$, the infinite part of the fiber of $\phi^{-1}(V)$ is contained in the scheme cut out by the forms $\tilde{g_j}(p_1(x), \ldots, p_d(x);x)$; note this is independent of $y$.  Since $pr_2(V) = pr_2(\phi^{-1}(V))$, the infinite part of the fiber of $\phi^{-1}(V)$ is also contained in $\oline{pr_2(V)} \cap \bb{P}^{q-1}_\infty$.  

Now $\oline{pr_2(V)} \cap \bb{P}^{q-1}_\infty$ has dimension $\leq d+t-1$.  For sufficiently general $p_i$, intersecting $\oline{pr_2(V)} \cap \bb{P}^{q-1}_\infty$ with all of the $\tilde{g}_j(p_1(x), \ldots, p_d(x);x)$ imposes $d$ independent conditions, hence the intersection has dimension $\leq t-1$ \cite[Prop.~1.7]{Sus}.  Hence the infinite part of the fiber has dimension $\leq t-1$, so the fiber itself has dimension $\leq t$.  \end{proof}

\begin{proof}[Proof of Theorem $\ref{projective suslin}$.] We write $W^i := W \cap (\Delta^i \times \bb{P}^d \times \bb{A}^q)$ for the intersection with a dimension $i$ face $\Delta^i \into \Delta^n$.

Let $G_1(t_i, Y, x), \ldots, G_s(t_i, Y,x)$ be equations defining $W$, and suppose they are chosen so that the forms $\tilde{g}_j := \lf_x \lf_{t,y} g_j$ generate the ideal of leading forms, i.e., cut out the subscheme obtained by taking the closure of $W |_{\bb{A}^d_0} \into \Delta^n \times \bb{A}^d_0 \times \bb{A}^q$ in $\bb{P}^{n+d} \times \bb{A}^q$, intersecting with $\bb{P}^{n+d-1}_\infty \times \bb{A}^q$, and then taking the closure via $\bb{A}^q \subset \bb{P}^q$ and intersecting with $\bb{P}^{q-1}_\infty$.

\textit{Definition of $\phi_0$.}  We construct the morphism $\phi_0 : \Delta^0 \times \bb{P}^d \times \bb{A}^q \to \Delta^0 \times \bb{P}^d \times \bb{A}^q$ as the product of the identity on $\Delta^0$ with a morphism $\phi : \bb{P}^d \times \bb{A}^q \to \bb{P}^d \times \bb{A}^q$.  Since $W$ intersects all the vertices in codimension at least $q$, we may apply Lemma $\ref{move on vertices}$ with $V = \cup_{\Delta^0 \into \Delta^n} W^0$ and $t=0$.  Thus we obtain the desired morphism on the $0$-simplices:
$$\phi( [ Y_0 : \cdots : Y_d], x) = ([Y_0 : Y_1 + p_{0,1}(x) Y_0 : \cdots : Y_d + p_{0,d}(x) Y_d], x),$$
where the $p_{0,k}$ are general homogeneous forms with degree large compared to the $x$-degree of the equations defining $W$. 

\textit{Definition of $\phi_1$ (model for induction step).} To construct the morphism $\phi_1$, we choose generic homogeneous forms $a_{1, ij}, p_{1,k}$ for $0 \leq i < j \leq n$ and $1 \leq k \leq d$, all of the same degree $m$, with $m > \deg (p_{0,k})$.  To give a formula for $\phi_1$ we introduce some more notation.  Let $[n]$ denote the ordered set $\{0 ,\ldots, n \}$.  For $\phi_1$ we will use $T_1 = \sum_{i <j \in [n]} t_i t_j = t_0t_1 + t_0t_2 + \cdots + t_{n-1}t_n$; for general $1 \leq r \leq n$ we set 
$T_r = \sum_{i_0 < \cdots < i_r \in [n] } t_{i_0} \cdots t_{i_r}$.

Now we define $\phi_1(t_i, Y, x) = (s_i, Z, x)$, where:
\begin{itemize}
\item $s_i = t_i(1 + \sum_{j \neq i}t_j a_{1, ij}(x) )$, with $a_{1, ij} = - a_{1, ji}$; and
\item $Z_0 = Y_0$, $Z_k = Y_k + (p_{0,k}(x) + T_1 p_{1,k}(x))Y_0$ for $1 \leq k \leq d$.
\end{itemize}
(The condition $a_{1, ij} = - a_{1, ji}$ ensures $\sum s_i =1$.)  Since $T_1=0$ and $s_i = t_i$ at every vertex, it follows that $\phi_1$ agrees with $\phi_0$ on the vertices.  Following the method used in Lemma $\ref{move on vertices}$, we calculate the fiber dimension of $\phi_1^{-1}(W)$ along the dimension 1 faces $\Delta^1 \into \Delta^n$.  Without loss of generality we work on the face $\Delta^1 \into \Delta^n$ defined by $t_2 = t_3 = \cdots = t_n = 0$.  Then we have:
$$\phi_1(t_0,t_1, Y_k ,x) = (t_0 + t_0t_1a_{1, 01}(x), t_1 - t_1 t_0 a_{1, 01}(x), Y_k + (p_{0,k}(x) + t_0t_1p_{1,k}(x) )Y_0, x).$$

\textit{$1$-simplices: finite part.} First we work on the finite part $\Delta^1 \times \bb{A}^d_{0} \times \bb{A}^q$.  Then, because $\deg (p_{1,k}) = \deg(a_{1, ij}) >\deg(p_{0,k}) > \deg_x (g_j)$, we have:
$$\lf_x( \phi_1^*(g_j |_{t_2= \cdots = t_n=0} ) ) = {(t_0 t_1)}^{\deg_{t,y} g_j} \tilde{g_j} (a_{1, 01}, - a_{1, 01}, p_{1,1}, \ldots, p_{1,d}, x).$$

Since $W^1$ has dimension $\leq d+1$, clearly $\oline{pr_3(W^1)}$ has dimension $\leq d+1$.  Since $\phi_1$ is an $\bb{A}^q$-morphism, we have $\oline{pr_3(W^1)} = \oline{pr_3(\phi_1^{-1}(W^1))}$.  Therefore $\oline{pr_3(\phi_1^{-1}(W^1)) } \cap \bb{P}^{q-1}_\infty$ has dimension $\leq d$.  Away from the vertices, our calculation of the leading $x$-term shows we can impose $d+1$ further independent conditions by the choices of $a_{01}, p_1, \ldots ,p_d$.  Hence we find, away from $t_0t_1 = 0$, the infinite part of the fiber is empty.  Hence the fiber itself is finite away from the vertices.  The fiber is finite on the vertices as well since $\phi_1 |_{\Delta^0} = \phi_0$.  Therefore, for every face $\Delta^1 \into \Delta^n$, the fibers of $\phi_1^{-1}(W^1) \to \Delta^1 \times \bb{P}^d$ are finite over $\Delta^1 \times \bb{A}^d_{0}$.  This verifies property (2).

\textit{$1$-simplices: infinite part.} Now we work on the hyperplane at infinity.  By Lemma $\ref{hyperplane lowers filtration}$, we can find a support $W'$ belonging to the smaller condition, and satisfying $W \cap (\Delta^1 \times \bb{P}^{d-1}_\infty \times \bb{A}^q) = W' \cap  (\Delta^1 \times \bb{P}^{d-1}_\infty \times \bb{A}^q)$.  Let $\bb{A}^M$ be the affine space parameterizing choices $(\underline{a}, \underline{p})$ giving rise to moving morphisms $\phi_{1, \underline{a}, \underline{p}}$.  Consider the family:
$$\xym{ \widetilde{\phi_1^{-1}(W^{'1})} \ar[r] \ar[rd] & \Delta^1 \times \bb{P}^d \times \bb{A}^q \times \bb{A}^M \ar[d] \\ & \bb{A}^M }$$
whose fiber over $\underline{a}, \underline{p} \in \bb{A}^M(k)$ is the subscheme ${(\phi_{1, \underline{a}, \underline{p}})}^{-1}(W^{'1}) \into \Delta^1 \times \bb{P}^d \times \bb{A}^q$.  By hypothesis, we have that ${(\phi_{1, \underline{0}, \underline{0}})}^{-1}(W^{'1}) = W^{'1}$ belongs to the smaller support condition.  Since our support conditions are open in algebraic families $ \ref{elem support properties}(\ref{elem support properties 4})$, it holds that ${(\phi_{1, \underline{a}, \underline{p}})}^{-1}(W^{'1})$ belongs to the smaller support condition for most choices.  Our moved support $\phi_1^{-1}(W^1)$ is quasi-finite over $\Delta^1 \times \bb{A}^d_0$ and agrees with $\phi_1^{-1}(W^{'1})$ over $\Delta^1 \times \bb{P}^{d-1}_\infty$, hence the moved support belongs to the smaller support condition for general choices of $\underline{a}, \underline{p}$.  This verifies property (3).

\textit{Definition of $\phi_n$ (general form of induction step).} We proceed inductively, using the evident notation for the components of the morphism: $\phi_r(t_0, \ldots, t_n, [Y_0 : \cdots : Y_d], \underline{x}) = (\phi_r(t_0), \ldots, \phi_r(t_n), [Y_0 : \phi_r(Y_1) : \cdots :\phi_r(Y_d)], \underline{x}).$  So far we have defined:
\begin{itemize}
\item $\phi_0(t_i) = t_i, \phi_0(Y_k) = Y_k + p_{0,k}(x)Y_0$;
\item $\phi_1(t_i) = t_i + \sum_{j \neq i}t_i t_j a_{1, ij} = \phi_0(t_i) + \sum_{j \neq i}t_i t_j a_{1, ij}$, subject to $a_{ij} + a_{ji} = 0$; and
\item $\phi_1(Y_k) = Y_k + (p_{0,k}(x) + T_1 p_{1,k}(x) )Y_0 = \phi_0(Y_k) +  T_1 p_{1,k}(x) Y_0$.
\end{itemize}
Then we set:
\begin{itemize}
\item $\phi_2(t_i) = \phi_1(t_i) + \sum_{\substack{j < k \\ i \neq j,k}} t_i t_j t_k a_{2, ijk}$, subject to $a_{ijk} +a_{jik} + a_{kij} = 0$ if $i < j < k$; and
\item $\phi_2(Y_k) = \phi_1(Y_k) +  T_2 p_{2,k}(x) Y_0$.
\end{itemize}

The morphism $\phi_r$ is constructed from $\phi_{r-1}$ just as $\phi_2$ was constructed from $\phi_1$.  For completeness we produce formulas for the $\phi_r$'s:
\begin{itemize}
\item $\phi_r(t_i) = \phi_{r-1}(t_i) + \sum_{\substack{i_1 < \cdots < i_r  \\ i \neq i_1, \ldots , i_r}} t_i t_{i_1} \cdots t_{i_r} a_{r, i i_1 \cdots i_r}$, subject to  $\sum_{s=0, \ldots r} a_{i_s \hat{I_s}} = 0$; and
\item $\phi_r(Y_k) = \phi_{r-1}(Y_k) + T_rp_{r,k}(x)Y_0.$
\end{itemize}
Here the $a_{r, I}, p_{r,k}$ are homogeneous of the same degree, with this degree being larger than $\deg (a_{r-1, I}) = \deg(p_{r-1,k})$; and $\hat{I_s}$ denotes the sequence $i_0 < \cdots < i_r$ skipping $i_s$.

The verification of the properties for $\phi_r$ proceeds as it did for $\phi_1$.  First we claim $\phi_r$ agrees with $\phi_{r-1}$ along all faces $\Delta^{r-1} \into \Delta^r \into \Delta^n$.  Now $\Delta^{r-1} \into \Delta^n$ is defined by the vanishing of $n-r+1$ of the $n+1$ coordinate functions.  Every term in $T_r$ is a product of $r+1$ coordinate functions; similarly for the difference $\phi_r(t_i) - \phi_{r-1}(t_i)$.  By the pigeonhole principle every term in $T_r$ vanishes along every face $\Delta^{r-1} \into \Delta^n$; similarly for $\phi_r(t_i) - \phi_{r-1}(t_i)$.  Hence the $\{ \phi_i \}$ define a pseudo-endomorphism.

Now we calculate the leading $x$-term of $\phi_r^*(g_j)$ along some face $\Delta^r \into \Delta^n$, say defined by $t_{r+1} = \cdots = t_n = 0$.  We find that $\lf_x (\phi_r^*(g_j |_{t_{r+1} = \cdots = t_n = 0}))$ is equal to
$${(t_0 \cdots t_r)}^{\deg_{t,y}g_j} \tilde{g_j}(a_{r, 01\cdots r}, a_{r, 102 \cdots r}, \ldots , a_{r, r01 \cdots r-1}, p_{r,1}, \ldots p_{r,d}, x).$$
The constraint $a_{0 \cdots r} + a_{10 \cdots r} + \cdots a_{r1 \cdots r-1} = 0$ means there are $r$ moduli among the $a$'s, and $d$ among the $p$'s.  Hence by choosing these appropriately we can, on the locus $t_0 \cdots t_r \neq 0$, slice the infinite part of the fiber $\oline{pr_3(W^r)}_\infty$ (which has dimension at most $d+r-1$) until it is empty.  Hence the fibers themselves are finite.  On $t_0 \cdots t_r =0$, the fibers are finite by induction.  We conclude $\phi_r^{-1}(W^r)$ is quasi-finite over $\Delta^r \times \bb{A}^d_0$, as property (2) requires.  To verify (3), we choose a support $W'$ belonging to the smaller support condition and agreeing with $W$ over $\Delta^n \times \bb{P}^{d-1}_\infty$.  The openness property $\ref{elem support properties}(\ref{elem support properties 4})$ of the conditions $S^{(q), e,f}$ guarantees most moving morphisms keep $W'$ in the same support condition, hence most morphisms carry $W$ to a smaller support condition. \end{proof}

\begin{theorem} \label{projective suslin homotopy} Let the notation and hypotheses be as in Theorem $\ref{projective suslin}$.  There exists a homotopy $\Phi_\bullet$ relating $\phi_\bullet$ to the identity, i.e., there are $\bb{A}^q$-morphisms (for $i=0, \ldots ,n$) $\Phi_i :\Delta^i \times \A \times \bb{P}^d \times \bb{A}^q \to \Delta^i \times \A \times \bb{P}^d \times \bb{A}^q$ such that $i_0 \circ \phi_i = \Phi_i \circ i_0$ and $i_1 = \Phi_i \circ i_1$ as morphisms $\Delta^i \times \bb{P}^d \times \bb{A}^q \to \Delta^i  \times \A \times \bb{P}^d \times \bb{A}^q$.  This homotopy can be chosen so that it enjoys the following properties.

\begin{enumerate}
\item The support $\Phi_n^{-1}(p^{-1} (W)) \into \Delta^n \times \A \times \bb{P}^d \times \bb{A}^q$ is quasi-finite over $\Delta^n \times (\A \setminus 1) \times \bb{A}^d_{0} $.
\item The support $\Phi_n^{-1}(p^{-1}(W))$ belongs to $S_n^{(q),e,f}(\A \times \bb{P}^d )$.
 \item The support ${(f_j \times 1 \times 1)}^{-1}(\Phi_n^{-1}(p^{-1}(W)))$ belongs to $S_n^{(q),e,f}(\bb{P}^d)$ for any morphism of the form $f_j : \Delta^n \to \Delta^n \times \A$.
\end{enumerate}
Here $p : \Delta^\bullet \times \A \times \bb{P}^d \times \bb{A}^q \to \Delta^\bullet \times \bb{P}^d \times \bb{A}^q$ denotes the projection. \end{theorem}

\begin{proof} The homotopy between $\phi_\bullet$ and the identity, $``t \id + (1-t) \phi,"$ is defined as follows.  The value of $\Phi_r(t,Y,t,x)$ is
$$(\ldots, t t_i + (1-t) \phi_r(t_i), \ldots, [Y_0 : t Y_1 + (1-t) \phi_r(Y_1) : \cdots : t Y_d + (1-t) \phi_r(Y_d)], x).$$
Then $\Phi_r(t,Y,1,x) = \id : \Delta^n \times \bb{P}^d \times \bb{A}^q \to \Delta^n \times \bb{P}^d \times \bb{A}^q$ and $\Phi_r(t,Y,0,x) = \phi_r(t,Y,x)$.  One verifies directly (using that, along the faces, $\phi_r$ restricts to $\phi_{r-1}$ and coordinates restrict to coordinates) that $\Phi_r$ restricts to $\Phi_{r-1}$ along the faces.

Having established that $\Phi_\bullet$ is a homotopy relating $\phi_\bullet$ to the identity, we proceed to check the first property, namely the quasi-finiteness of $\Phi_n^{-1}(p^{-1}(W))$ over $\Delta^n \times (\A \setminus 1) \times \bb{A}^d_0$.  As in the proof of Theorem $\ref{projective suslin}$, we calculate the leading $x$-term of $\Phi_r^*(g_j)$ along the face defined by $t_{r+1}= \cdots = t_n =0$.  We find that $\lf_x (\Phi_r^*(g_j |_{t_{r+1} = \cdots = t_n = 0}))$ is equal to
$${((1-t)(t_0 \cdots t_r))}^{\deg_{t,y}g_j} \tilde{g_j}(a_{r, 01\cdots r}, a_{r, 102 \cdots r}, \ldots , a_{r, r01 \cdots r-1}, p_{r,1}, \ldots p_{r,d}, x).$$
Therefore the fibers are quasi-finite over $(\Delta^r \setminus V(t_0 \cdots t_r)) \times (\A \setminus 1) \times \bb{A}^d_0$.  Over $V(t_0 \cdots t_r) \times (\A \setminus 1) \times \bb{A}^d_0$ we apply induction, hence the property (1) holds.

The support $p^{-1}(W)$ belongs to $S_n^{(q),e,f}(\A \times \bb{P}^d)$ since our support conditions are compatible with flat pullback.  Viewing $\Phi_n$ as determined by a point $(\underline{a}, \underline{p})$ in an affine space, and noticing that $(\underline{0}, \underline{0})$ recovers $p^{-1}(W)$ and hence satisfies (2), we conclude a general choice of $\Phi_n$ satisfies (2) by the openness property $\ref{elem support properties}(\ref{elem support properties 4})$ of the conditions $S^{(q)e,f}(\bb{P}^d)$.  For (3), we notice that pulling back by $f_{1_n}$ gives the original support $W$.  Every other embedding $f_j: \Delta^n \to \Delta^n \times \A$ has smaller incidence with the non-quasi-finite locus of $\Phi_n^{-1}(p^{-1}(W))$, hence we obtain (3).  \end{proof}

\subsection{Extension to smooth projective varieties}

\begin{lemma} \label{hyperplane lowers filtration general} Let $X$ be a smooth equidimensional $k$-scheme that is projective and of dimension $d$, and let $\{ Z^n_i \} \subset S_n^{(q),e,f}(X)$ be a finite subfamily of supports.  Then there exists a very ample divisor $H \into X$ and supports $\{ {Z^n_i}' \} \subset S_n^{(q)}(X)$ satisfying:
\begin{itemize}
\item $Z^n_i \cap {Z^n_i}' \supseteq Z^n_i \cap H$; and
\item ${Z^n_i}'$ belongs to $S_n^{(q),e,f+1}(X)$, if $f \leq d+n -1$; or to $S_n^{(q),e-1,0}(X)$, if $f=d+n$.
\end{itemize}
If $(H ,\{ {Z^n_i}' \})$ satisfies the above properties, then so does $(H_1 ,\{ {Z^n_i}' \})$ for $H_1$ a general element of any linear system to which $H$ belongs.
\end{lemma}

\begin{proof} For simplicity of notation we assume $Z = Z^n_i$ is a single subvariety.  There is a finite flat morphism $f : X \to \bb{P}^d$, so we may apply Corollary $\ref{hyperplane lowers filtration}$ to the support $W=f(Z)$.  Thus we find a hyperplane $H \into \bb{P}^d$ and a support $W'$ such that $f(Z) \cap H = W' \cap H$, and such that $W'$ belongs to a smaller support condition.  The pullbacks of $H$ and $W'$ to $X$ satisfy the conclusions of the lemma. \end{proof} 

For convenience of reference we record the following consequence of Zariski excision.

\begin{lemma} \label{easy Zariski} Let $X \subset \oline{X}$ be an open immersion of smooth $k$-schemes, with $X_\infty := \oline{X} \setminus X$.  Let $Z_1, Z_2 \in S^{(q)}(\oline{X})$ be supports satisfying $Z_1 \into Z_2$.  Let $Z_i^o$ denote the induced support in $S^{(q)}(X)$, and let $Z_i^\infty$ denote the intersection $Z_i \cap (\Delta^\bullet \times X_\infty \times \bb{A}^q)$.  Suppose $Z_3 \in S^{(q)}(\oline{X})$ is a support such that $Z_3^\infty \supseteq Z_2^\infty$.  Suppose $E \in {\bf{Spt}}(k)$ satisfies Zariski excision.  Then the canonical map
$$E^{Z_2 \setminus Z_1 \cup (Z_2 \cap Z_3)}(\oline{X} \setminus Z_1 \cup Z_3) \to E^{Z_2^o \setminus Z_1^o \cup {(Z_2 \cap Z_3)}^o}(X \setminus {(Z_1 \cup Z_3)}^o)$$
is a weak equivalence. \end{lemma}
\begin{proof} We have ${Z_2 \setminus Z_1 \cup (Z_2 \cap Z_3)} = Z_2^o \setminus Z_1^o \cup {(Z_2 \cap Z_3)}^o$.  The closed immersion $Z_2 \setminus Z_1 \cup (Z_2 \cap Z_3) \into \Delta^\bullet \times \oline{X} \times \bb{A}^q \setminus Z_1 \cup Z_3$ factors through $\Delta^\bullet \times X \times \bb{A}^q \setminus Z_1^o \cup Z_3^o$. \end{proof}

\begin{proposition} \label{projective space we} Let $S_1, S_2$ be successive support conditions in the tower $\ref{support tower}$.  Suppose $E \in {\bf{Spt}}(k)$ satisfies Zariski excision. Then the canonical morphism $E^{S_1}(\bb{P}^d) \to E^{S_2}(\bb{P}^d)$ is a weak equivalence. \end{proposition}

\begin{proof} The proof is modeled on Steps 2 and 3 of the proof of Proposition $\ref{affine tower collapse}$, with the hyperplane playing the role of the exceptional locus.  There is no need to treat induced supports separately, so as in Step 2 we suppose given $Z_1 \subset S_1(\bb{P}^d), Z_2 \subset S_2(\bb{P}^d)$, and assume $Z_1 \into Z_2$.  We will find $Y_1 \subset S_1(\bb{P}^d), Y_2 \subset S_2(\bb{P}^d)$ such that the induced map $E^{Z_2 \setminus Z_1 }(\bb{P}^d \setminus Z_1) \to E^{Z_2 \cup Y_2 \setminus Z_1 \cup Y_1}(\bb{P}^d \setminus Z_1 \cup Y_1)$ is nullhomotopic.  

We continue with the analogue of Step 2.  By Corollary $\ref{hyperplane lowers filtration}$, we may choose a hyperplane $H \into \bb{P}^d$ and a support $Z_2' \subset S_1(\bb{P}^d)$ which agrees with $Z_2$ along $H$.  Then we may apply Theorem $\ref{projective suslin}$ with $W = Z_2$ to obtain a pseudo-endomorphism $\phi_\bullet$ of the $n$th truncation of $\Delta^\bullet \times \bb{P}^d \times \bb{A}^q$ such that $\phi_n^{-1}(Z_2) \subset S_1(\bb{P}^d)$.  We have the following commutative diagram, in which we use the notations $Z_i^o := Z_i \setminus Z_i \cap H, \bb{A}^d = \bb{P}^d \setminus H, Z_3 := \oline{F^*Z_1^o} \cup \oline{\phi^{-1}(Z_2^o)}\cup \oline{F^*{Z_2'}^o}, Z_3^o := F^*Z_1^o \cup \phi^{-1}(Z_2^o) \cup F^* Z_2'$.
$$\xym{
E^{Z_2 \setminus Z_1 \cup (Z_2 \cap Z_2')}(\bb{P}^d \setminus Z_1 \cup Z_2' ) \ar[r] \ar[d] & E^{\oline{F^*Z_2^o} \setminus \oline{F^*Z_2^o} \cap Z_3} (\bb{P}^d \setminus Z_3 ) \ar[d] \\
E^{Z_2^o \setminus Z_1^o \cup(Z_2^o \cap {Z_2'}^o)}( \bb{A}^d \setminus {(Z_1 \cup Z_2')}^o) \ar[r] & E^{F^*Z_2^o \setminus F^*Z_2^o \cap Z_3^o} (\bb{A}^d \setminus Z_3^o ) \\}$$

The homotopy constructed in Theorem $\ref{projective suslin homotopy}$ gives back the homotopy used in Proposition $\ref{no new supports}$ and Corollary $\ref{nullhomotopy pairs}$.  Thus, as in Step 3, the bottom arrow is a nullhomotopy.  The vertical arrows are weak equivalences by Lemma $\ref{easy Zariski}$.  Therefore the arrow in the top row is a nullhomotopy, hence so is $E^{Z_2 \setminus Z_1}(\bb{P}^d \setminus Z_1) \to E^{\oline{F^*Z_2^o} \setminus \oline{F^*Z_2^o} \cap Z_3}(\bb{P}^d \setminus Z_3)$.  Thus we just need to show this map appears in the colimit which is the cofiber of the map $E^{S_1}(\bb{P}^d) \to E^{S_2}(\bb{P}^d)$, i.e., we need $Z_3 = \oline{F^*Z_1^o} \cup \oline{\phi^{-1}(Z_2^o)} \cup \oline{F^*{Z_2'}^o} \subset S_1(\bb{P}^d)$ and $\oline{F^*Z_2^o} \subset S_2(\bb{P}^d)$.

We verify this directly, with no need for further shrinking as in Step 4.  Now we recall that $F^*(-)$ is the support condition generated by ${(f_j \times 1 \times 1)}^{-1}(\Phi^{-1}(p^{-1}(-)))$.  Since $Z_2' \subset S_1(\bb{P}^d)$, we obtain that $\oline{F^*{Z_2'}^o} \subset S_1(\bb{P}^d)$ by Theorem $\ref{projective suslin homotopy}$(3).  Similarly, $Z_1 \subset S_1(\bb{P}^d)$ implies $\oline{F^*Z_1^o} \subset S_1(\bb{P}^d)$; and $\oline{F^*Z_2^o} \subset S_2(\bb{P}^d)$.  The support $\oline{\phi^{-1}(Z_2^o)}$ belongs to $S_1(\bb{P}^d)$ by Theorem $\ref{projective suslin}$(3).  This completes the proof.  \end{proof}

\begin{corollary} If $E \in {\bf{Spt}}(k)$ satisfies Zariski excision, then the map $\alpha_{\bb{P}^d} : E^Q(\bb{P}^d) \to E^{(q)}(\bb{P}^d)$ is a weak equivalence.  \end{corollary}

We isolate the key fact used in the proof of Proposition $\ref{projective space we}$.

\begin{proposition} \label{projective space we key} Let the notation and hypotheses be as in Theorem $\ref{projective suslin}$ and the proof of Proposition $\ref{projective space we}$.  Then the canonical morphism
$$E^{Z_2 \setminus Z_1}(\bb{P}^d \setminus Z_1) \to E^{\oline{F^*Z_2^o} \setminus \oline{F^*Z_2^o} \cap Z_3} (\bb{P}^d \setminus Z_3)$$ is nullhomotopic. \end{proposition}

We will need the analogue of Proposition $\ref{birationality lemma}$ for a smooth projective variety.

\begin{proposition} \label{projective birationality lemma} Let $X$ be a smooth equidimensional $k$-scheme that is projective and of dimension $d$.  Let $\{ Z^n_i \} \subset S^{(q),e,f}_n(X)$ be a finite subfamily of supports, and suppose none are induced.  Then there exists a finite morphism $f: X \to \bb{P}^d$ and supports $\{ {Z^n_i}' \} \subset S^{(q)}_n(X)$ satisfying:
\begin{enumerate}
\item the restriction of $1 \times f \times 1$ to $Z^n_i$ is birational onto its image;
\item $Z^n_i \cap {Z^n_i}'$ contains the exceptional locus $\Exc (1 \times f \times 1 |_{Z^n_i})$; and  
\item ${Z^n_i}' \in S_n^{(q),e,f+1}(X)$, if $f \leq d+n-1$; or ${Z^n_i}' \in S_n^{(q),e-1,0}(X)$, if $f = d+n$.
\end{enumerate}
\end{proposition}

\begin{proof} By Lemma $\ref{hyperplane lowers filtration general}$ we can find an ample divisor $H \into X$ such that $\{ Z^n_i \cap H \}$ satisfies the smaller support condition, except that it does not dominate $\Delta^n \times X$.  We can apply Proposition $\ref{birationality lemma}$ to the family $\{Z^n_i \setminus Z^n_i \cap H \} \subset S^{(q),e,f}_n( X \setminus H)$ to obtain a morphism $f : X \setminus H \to \bb{A}^d$ which is birational on $Z^n_i \setminus Z^n_i \cap H$ and for which the exceptional locus satisfies the smaller support condition, except that it does not dominate $\Delta^n \times (X \setminus H)$.  Now if the morphism $f$ extends to a finite morphism $\oline{f} : X \to \bb{P}^d$, then since
$$\Exc(1 \times \oline{f} \times 1 |_{Z^n_i}) \subseteq \Exc(1 \times f \times 1 |_{Z^n_i \setminus Z^n_i \cap H}) \cup (Z^n_i \cap H)$$
and the right hand side lies in the smaller support condition, the image of the exceptional locus extends to a dominant support satisfying the smaller support condition by Lemma $\ref{extend from divisor}$.  Then we can take for ${Z^n_i}'$ the pullback via $f$ of this extension.  Thus we need to perturb the situation so that $f$ extends, without disturbing the other properties.

The morphism $f : X \setminus H \to \bb{A}^d$ is determined by $d$ functions $f_1, \ldots, f_d \in \Gamma(X \setminus H, \cal{O}_{X \setminus H})$.  There exists a positive integer $m_0$ such that all of these functions extend to global sections of the line bundle $\cal{O}_X(m_0 H)$.  Thus we have a diagram:
$$\xym{X \setminus H \ar[r] \ar[d]^f & X \ar@{-->}[d] \ar[r] & \bb{P}(H^0(X,\cal{O}_X(m_0H))) \ar@{-->}[ld] \\
\bb{A}^d \ar[r] & \bb{P}(Q) \\}$$
where $Q$ is a $(d+1)$-dimensional quotient of the $k$-vector space $H^0(X,\cal{O}_X(m_0H))$, and the dashed arrows are rational maps.  Set $n+1 = \dim_k H^0(X,\cal{O}_X(m_0H))$.

Let $V_{i}$ denote the variety of $i$-codimensional subspaces of the $k$-vector space $H^0(X,\cal{O}_X(m_0H))$.  We may view $V_{d+1}$ as the variety of linear subspaces $\bb{P}^{n-d-1} \into \bb{P}(H^0(X,\cal{O}_X(m_0H)))$ of dimension $(n-d-1)$, and, by taking the subspace as the center for a linear projection, as the variety of rational maps $\bb{P}(H^0(X,\cal{O}_X(m_0H))) \dashrightarrow \bb{P}^d$.  We may view $V_1$ as the variety of hyperplanes in $\bb{P}(H^0(X,\cal{O}_X(m_0H)))$.

Now let $I \into V_{1} \times V_{d+1}$ denote locus of pairs $(h,c)$ such that $[h]$ contains $[c]$ as a subspace.  Then $I \to V_{d+1}$ is a $\bb{P}^d$-bundle.  We set $C := \ker (H^0(X, \cal{O}_X (m_0 H)) \to Q)$ and use lowercase letters $h,c$ for variables.  Then the left hand square of the diagram (the morphism $f$ and its rational extension) determines the point $(H,C) \in I$, for which $\Exc(1 \times f \times 1 |_{Z_i \setminus Z_i \cap H})$ and  $(Z_i \cap H)$ both belong to the smaller support condition.  (To be more precise, since these supports are not dominant, we should say they extend to dominant supports in the smaller support condition.)  This is an open condition: for $(h,c)$ in a Zariski neighborhood $U \subset I$ of $(H,C)$, we have that $Z_i \cap [h]$ lies in the smaller support condition, and projection away from $[c]$ induces a morphism $X \setminus [h] \to \bb{A}^d$ which is birational on $Z_i \setminus Z_i \cap [h]$ and has exceptional locus in the smaller support condition.  Furthermore there is an open dense subset $U_{d+1} \subset V_{d+1}$ corresponding to centers $c$ for which the induced map $X \to \bb{P}([q])$ is a finite morphism.  Now we just take any point in the intersection $p_2^{-1}(U_{d+1}) \cap U \subset I$, and the corresponding finite morphism satisfies the desired properties.  \end{proof}

At this point we can mimic the proof of Proposition $\ref{affine tower collapse}$.

\begin{theorem} Let $X$ be a smooth equidimensional $k$-scheme that is projective and of dimension $d$.  Suppose $E \in {\bf{Spt}}(k)$ satisfies Nisnevich excision.  Let $S_1, S_2$ be successive support conditions in the tower $\ref{support tower}$.  Then the canonical map $E^{S_1}(X) \to E^{S_2}(X)$ is a weak equivalence.
 \end{theorem}

\begin{proof} Suppose given $Z_1 \subset S_1(X), Z_2 \subset S_2(X)$, and suppose $Z_1 \into Z_2$.  Induced supports are handled in the same manner as in the proof of Proposition $\ref{affine tower collapse}$, so we assume none of the supports are induced.  We aim to find $Y_1 \subset S_1(X), Y_2 \subset S_2(X)$ such that $E^{Z_2 \setminus Z_1}(X \setminus Z_1) \to E^{Z_2 \cup Y_2 \setminus Z_1 \cup Y_1}(X \setminus Z_1 \cup Y_1)$ is a nullhomotopy.

Now choose a very ample divisor $H$ as in Lemma $\ref{hyperplane lowers filtration general}$, and a finite morphism $f : X \to \bb{P}^d$ as in Proposition $\ref{projective birationality lemma}$ for the family $Z_2$.  Let $Z_2' \subset S_1(X)$ denote the union of the supports obtained in Lemma $\ref{hyperplane lowers filtration general}$ and Proposition $\ref{projective birationality lemma}$.  Then $Z_2 \cap Z_2' \supseteq (Z_2 \cap H) \cup \Exc (1 \times f \times 1|_{Z_2})$.  By Nisnevich excision and cofiber comparison (as in Step 2 of the proof of Proposition $\ref{affine tower collapse}$), we deduce a weak equivalence
$$E^{f(Z_2) \setminus f( Z_1 \cup (Z_2 \cap Z_2'))}(\bb{P}^d \setminus f(Z_1 \cup Z_2') ) \xrightarrow{\sim} E^{Z_2 \setminus Z_1 \cup (Z_2 \cap Z_2')}(X \setminus Z_1 \cup Z_2').$$

Now set $Z = Z_2, W= f(Z), Z' = Z_1 \cup Z_2'$, and $W' := f(Z') \cup W^e$.  Here $W^e$ is an extension of $W \cap H$ to a dominant support which is quasi-finite away from $H$ and hence belongs to the support condition $S_1(\bb{P}^d)$. We apply Proposition $\ref{projective space we key}$ to conclude the canonical map
$$E^{W \setminus W' \cap W}(\bb{P}^d \setminus W' ) \to E^{\oline{F^*W^o} \setminus \oline{\phi^{-1}(W^o)} \cup \oline{F^*W'^o}} (\bb{P}^d \setminus \oline{\phi^{-1}(W^o)} \cup \oline{F^*W'^o}  )$$
is a nullhomotopy.  Therefore the map from $E^{Z \setminus Z' \cap Z}(X \setminus Z')$ to
$$E^{Z \cup f^{-1}(\oline{R^* W^o}) \setminus f^{-1}(\oline{\phi^{-1}(W^o)}) \cup f^{-1}(\oline{F^*W'^o}) \cup  (f^{-1}(\oline{R^*W}) \cap  Z^+) }(X \setminus f^{-1}(\oline{\phi^{-1}(W^o)}) \cup f^{-1}(\oline{F^*W'^o}) \cup  Z^+)$$
is also a nullhomotopy.  Finally by the same support-shrinking argument as in Step 4 of the proof of Proposition $\ref{affine tower collapse}$, we conclude the morphism from $E^{Z \setminus Z' \cap Z}(X \setminus Z')$ to
$$E^{Z \cup f^{-1}(\oline{R^* W^o}) \setminus f^{-1}(\oline{\phi^{-1}(W^o)}) \cup f^{-1}(\oline{F^*W'^o})}(X \setminus f^{-1}(\oline{\phi^{-1}(W^o)}) \cup f^{-1}(\oline{F^*W'^o}))$$
is a nullhomotopy.  We are finished provided we show this map occurs in the cofiber of the map $E^{S_1}(X) \to E^{S_2}(X)$.  So we must verify the following claims:
\begin{enumerate}
\item $Z \cup f^{-1}(\oline{R^* W^o}) \subset S_2(X)$,
\item $f^{-1}(\oline{\phi^{-1}(W^o)}) \subset S_1(X)$, and
\item $f^{-1}(\oline{F^*W'^o}) \subset S_1(X)$.
\end{enumerate}

The operation $Z \mapsto F^*Z$, hence also $Z \mapsto R^*Z$, preserves the tower $\ref{support tower}$.  The functoriality properties $\ref{elem support properties}(\ref{elem support properties 2},\ref{elem support properties 3})$ imply $f$ preserves the tower $\ref{support tower}$.  The claim (1) follows.

Since $Z_2 \subset S_2(X)$, we have $W \subset S_2(\bb{P}^d)$.  By Theorem $\ref{projective suslin}$(3), it follows that $\oline{\phi^{-1}(W^o)} \subset S_1(\bb{P}^d)$.  The claim (2) follows.

We have verified the components of $W'$ belong to the support condition $S_1(\bb{P}^d)$, hence we have $\oline{W'^o} = W' \subset S_1(\bb{P}^d)$.  Then (3) follows from the compatibility of the tower $\ref{support tower}$ with the operations $Z \mapsto F^*Z$ and $Z \mapsto f^* (f_* (Z))$.  This completes the proof. \end{proof}

\begin{corollary} \label{main comparison projective} Let $X$ be a smooth equidimensional projective $k$-scheme.  Suppose $E \in {\bf{Spt}}(k)$ satisfies Nisnevich excision.  Then $\alpha_X : E^Q(X) \to E^{(q)}(X)$ is a weak equivalence. \end{corollary}

\begin{remark} \label{qproj failure} Our methods fall short of establishing the result for a general smooth quasi-projective $X$ for the following simple reason.  In general we can find a quasi-finite map $f : X \to \bb{P}^d$, but the operation of pushing a cycle down, moving it on $\bb{P}^d$, then pulling it back, involves taking the closure of the cycle on $\bb{P}^d$.  The support conditions $S^{(q),e,f}(-)$ are not, in general, compatible with taking closures.  In particular there is no guarantee the closure of a cycle lying in some filtration level will even intersect all the faces properly.  Put differently, the quasi-finite morphism $f$ factors as an open immersion (say, $j : X \subset \oline{X}$) followed by a finite morphism, and there is no reason to expect closures of cycles in $S^{(q)}(X)$ to lie in $S^{(q)}(\oline{X})$, even if we are given a smooth compactification $\oline{X}$ of $X$. \end{remark}

\bibliography{tspecreferences}{}

\begin{thebibliography}{10}

\bibitem{Blhigher}
Spencer Bloch.
\newblock Algebraic cycles and higher {$K$}-theory.
\newblock {\em Adv. in Math.}, 61(3):267--304, 1986.

\bibitem{FS}
Eric~M. Friedlander and Andrei Suslin.
\newblock The spectral sequence relating algebraic {$K$}-theory to motivic
  cohomology.
\newblock {\em Ann. Sci. \'Ecole Norm. Sup. (4)}, 35(6):773--875, 2002.

\bibitem{BCC}
Eric~M. Friedlander and Vladimir Voevodsky.
\newblock Bivariant cycle cohomology.
\newblock In {\em Cycles, transfers, and motivic homology theories}, volume 143
  of {\em Ann. of Math. Stud.}, pages 138--187. Princeton Univ. Press,
  Princeton, NJ, 2000.

\bibitem{Jardine}
J.~F. Jardine.
\newblock Stable homotopy theory of simplicial presheaves.
\newblock {\em Canad. J. Math.}, 39(3):733--747, 1987.

\bibitem{MLChow}
Marc Levine.
\newblock Chow's moving lemma and the homotopy coniveau tower.
\newblock {\em $K$-Theory}, 37(1-2):129--209, 2006.

\bibitem{MLHtyCon}
Marc Levine.
\newblock The homotopy coniveau tower.
\newblock {\em J. Topol.}, 1(1):217--267, 2008.

\bibitem{ST}
Marc Levine.
\newblock Slices and transfers.
\newblock {\em Doc. Math.}, (Extra volume: Andrei A. Suslin sixtieth
  birthday):393--443, 2010.

\bibitem{SegalCats}
Graeme Segal.
\newblock Categories and cohomology theories.
\newblock {\em Topology}, 13:293--312, 1974.

\bibitem{Sus}
Andrei~A. Suslin.
\newblock Higher {C}how groups and etale cohomology.
\newblock In {\em Cycles, transfers, and motivic homology theories}, volume 143
  of {\em Ann. of Math. Stud.}, pages 239--254. Princeton Univ. Press,
  Princeton, NJ, 2000.

\bibitem{Thomason}
R.~W. Thomason.
\newblock Algebraic {$K$}-theory and \'etale cohomology.
\newblock {\em Ann. Sci. \'Ecole Norm. Sup. (4)}, 18(3):437--552, 1985.

\bibitem{TT}
R.~W. Thomason and Thomas Trobaugh.
\newblock Higher algebraic {$K$}-theory of schemes and of derived categories.
\newblock In {\em The {G}rothendieck {F}estschrift, {V}ol.\ {III}}, volume~88
  of {\em Progr. Math.}, pages 247--435. Birkh\"auser Boston, Boston, MA, 1990.

\end{thebibliography}
\bibliographystyle{plain}

\end{document}